\documentclass{amsart}

\usepackage{tikz}
\usepackage[]{mcode}
\usetikzlibrary{automata,positioning}
\usepackage{mathrsfs}
\usepackage{graphicx}
\usepackage{graphics}
\usepackage{color}
\usepackage{epstopdf}
\usepackage{algorithm}
\usepackage{algorithmic}
\usepackage{setspace}
\usepackage{circle}
\usepackage{amsmath}
\usepackage{amsfonts}
\usepackage{amssymb}
\usepackage{epstopdf}
\usepackage{pbox, hyperref}
\usepackage{caption}
\usepackage{amsmath}
\usepackage[font=footnotesize,labelfont=bf,compatibility=false]{caption}
\usepackage{subcaption}
\usepackage{pstool}
\usepackage{booktabs}
\newlength\figureheight
\newlength\figurewidth
\newtheorem{definition}{Definition}
\newtheorem{theorem}{Theorem}
\newtheorem{lemma}[theorem]{Lemma}
\newtheorem{proposition}[theorem]{Proposition}

\usepackage{stmaryrd}
\usepackage{tikz, pgfplots}
\pgfplotsset{compat=newest}
\usetikzlibrary{arrows, positioning, patterns}

\newcommand{\R}{{\mathbb{R}}}

\newcommand{\N}{{\mathbb{N}}}

\newcommand{\Pre}{\mathrm{Pre}}
\newcommand{\Post}{\textrm{Post}}

\newcommand\eqbir{\mathrel{\overset{\makebox[0pt]{\mbox{\normalfont\tiny\sffamily (1)}}}{=}}}
\newcommand\eqiki{\mathrel{\overset{\makebox[0pt]{\mbox{\normalfont\tiny\sffamily (2)}}}{=}}}
\newcommand\incbir{\mathrel{\overset{\makebox[0pt]{\mbox{\normalfont\tiny\sffamily (1)}}}{\subseteq}}}
\newcommand\inciki{\mathrel{\overset{\makebox[0pt]{\mbox{\normalfont\tiny\sffamily (2)}}}{\subseteq}}}

\newcommand\equivbir{\mathrel{\overset{\makebox[0pt]{\mbox{\normalfont\tiny\sffamily (1)}}}{\equiv}}}
\newcommand\equiviki{\mathrel{\overset{\makebox[0pt]{\mbox{\normalfont\tiny\sffamily (2)}}}{\equiv}}}
\newcommand\equivuc{\mathrel{\overset{\makebox[0pt]{\mbox{\normalfont\tiny\sffamily (3)}}}{\equiv}}}

\begin{document}
%
\title{Mode-Target Games: \\ Reactive synthesis for control applications}
%
\author{Ayca Balkan, Moshe Vardi, Paulo Tabuada
}
%
%
%
\maketitle
\begin{abstract}
In this paper we introduce a class of Linear Temporal Logic (LTL) specifications for which the problem of synthesizing controllers can be solved in polynomial time. The new class of specifications is an LTL fragment that we term \emph{Mode-Target} (MT) and is inspired by numerous control applications where there are modes and corresponding (possibly multiple) targets for each mode. We formulate the problem of synthesizing a controller enforcing an MT specification as a game and provide an algorithm that requires $O(\sum_i t_i n^2)$ symbolic steps, where $n$ is the number of states in the game graph, and $t_i$ is the number of targets corresponding to mode $i$.

\end{abstract}

\section{Introduction}
The results in this paper are developed under the correct-by-design philosophy for Cyber-Physical Systems (CPS) advocating control design methodologies that produce, not only the controller, but also a proof of its correctness. This design philosophy should be contrasted with the widely used design-and-verify approach under which a designer re-designs the controller to weed out the bugs that are found during multiple verification rounds. By placing greater emphasis and effort in the design phase it is possible to greatly reduce the verification efforts thereby reducing the design time and cost of complex CPS \cite{intro1, intro2, intro3, intro4}. 

The correct-by-design philosophy, however, is not without its own challenges and the purpose of this paper is to address one of the most critical: computational complexity. If one takes Linear Temporal Logic (LTL) as the specification formalism, it is known that synthesizing a controller enforcing such specifications is doubly exponential in the length of the formula.
%
This led several researchers to seek fragments of LTL that are small enough for the complexity of synthesis to be lower, yet large enough to be practically relevant ~\cite{fragment1,fragment2,GR1,GeneralizedRabin,wolf,fragment3}. Among these, the one that had the biggest practical impact was the Generalized Reactivity (1) fragment, abbreviated as GR(1), for which the controller synthesis can be solved in polynomial time in the size of the transition system \cite{GR1}. Even though the GR(1) fragment was not originally intended for control applications, several researchers demonstrated its usefulness to synthesize correct-by-design controllers in practical scenarios ~\cite{application1,application2}. Later, extending the ideas in \cite{GR1}, the Generalized Rabin (1) fragment was shown to be the largest class of LTL specifications for which the controller synthesis problem is still polynomial in the size of the transition system, unless P=NP \cite{GeneralizedRabin}. 

In this paper,  inspired by control applications, we introduce a new fragment of  LTL termed Mode-Target (MT). An MT formula describes a setting where there are modes and corresponding targets for each mode. When the system is in a certain mode, the specification requires the system to reach one of the possible targets for that mode and stay there as long as the mode does not change. If the mode changes, there is no obligation to reach or stay within  the target region of the previous mode. We use MT formulas to define mode-target games, a subclass of LTL games. The winning condition of an MT game is an MT formula and, moreover, the game graph conforms to additional restrictions on the structure of the modes. We believe that modeling the desired behavior of control systems in this way, via modes and targets, is quite natural for designers. We support this claim in Section \ref{problemformulation} by giving three concrete examples from different application domains that illustrate the usefulness of MT games. The first example is an adaptive cruise controller, whose specifications are outlined by the International Standardization Organization (ISO). The second example builds on \cite{jyoHSCC}, where researchers from the Toyota Technical Center described the desired behavior for an air-fuel-ratio controller in signal temporal logic. The third example is the control of certain chemicals inside a nuclear power plant during shutdown and startup operations as outlined in \cite{nuclearpowerplant}. We show that the controller synthesis problem for all of these examples can be posed as finding a winning strategy for an MT game.

The contributions of this work can be summarized as follows:
\begin{itemize}
\item We propose MT as a practically useful
LTL fragment from a modeling perspective. Doing so, we extend an earlier version of this work where a more restricted class of formulas was introduced as MT formulas \cite{ADHSBalkan}. We provide three concrete control applications as an illustration of the large class of problems that can be naturally modeled as MT games. 
\item We introduce the notion of simple games that abstracts the key properties of GR(1) and MT games so as to prove the correctness and complexity of the proposed algorithms in a transparent manner. In doing so, we provide a new and simpler proof for the correctness and complexity estimates of the existing controller synthesis algorithms for GR(1) while highlighting the commonalities and differences between GR(1) and MT games. In particular, we show that MT games are also GR(1) games.
\item We propose an algorithm to synthesize controllers enforcing MT specifications which requires $O(\sum_i t_i n^2)$ symbolic steps where $n$ is the number of states in the game graph and $t_i$ is the number of targets corresponding to mode $i$. In contrast, the complexity of the algorithm resulting from embedding MT games into GR(1) games and using existing synthesis algorithms for the GR(1) fragment is $O(\sum_i tn^2)$ where $t$ is the largest number of modes across all the targets. Although these two complexity upper bounds coincide when the number of targets for each mode is the same, we empirically show in Section \ref{experiments} that the proposed synthesis algorithm still outperforms the synthesis algorithm obtained via the GR(1) embedding in this situation.
\end{itemize}
%

The rest of the paper is organized as follows. In Section~\ref{preliminaries},
we review the syntax and semantics of LTL and introduce LTL games. We formally define MT games in Section~\ref{problemformulation} and illustrate their usefulness via examples from control. In Section~\ref{directMT} we present an algorithm for solving MT games. We then show in Section~\ref{algorithmsection} that every MT game can be formulated as a GR(1) game. This leads to an alternative solution for MT games via existing algorithms to solve GR(1) games. We experimentally compare the two algorithms for the solution of MT games in Section~\ref{experiments} and conclude with Section~\ref{conclusion}.

\section{Preliminaries}\label{preliminaries}
We start by reviewing the syntax and semantics of Linear Temporal Logic (LTL) and corresponding games.
\begin{subsection}{Linear Temporal Logic}

Consider a set of atomic propositions $P$. LTL formulas are constructed according to the following grammar:
\[\varphi::= p \in P\,\vert\,\neg\,\varphi\,\vert\,\varphi\,\vee\,\varphi\,
\vert\,\Circle\,\varphi\,\vert\,\varphi\, \mathcal{U}\,\varphi.\]
We denote the set $2^P$ by $\Sigma$, where $2^{P}$
is the set of all subsets of $P$. An \emph{infinite word} is an element of
$\Sigma^\omega$ where $\Sigma^\omega$ denotes the set of all infinite strings or words obtained by concatenating elements or letters in $\Sigma$.
We also regard elements $w \in \Sigma^\omega$ as maps $w:\N\to \Sigma$. Using this interpretation we denote $w(i)$ by $w_i$. In the context of LTL, the index $i$ models time and $w_i$ is interpreted as the set of  atomic propositions that hold at time $i$.
 
The semantics of an LTL formula $\varphi$ is described by a satisfaction relation $\models$ that defines when the string $w\in \Sigma^\omega$ satisfies the formula $\varphi$ at time $i\in \N$, denoted by $w,i\models \varphi$:
\begin{itemize}
\setlength\itemsep{0.4mm}
\item For $p \in P$, we have $w, i \models p$ iff $p \in w_i$,
\item $w, i \models \neg \varphi$ iff $w, i \not\models \varphi$,
\item $w, i \models \varphi \vee \psi$ iff $w, i \models \varphi$ or $w, i \models \psi$,
\item $w, i \models \Circle \varphi$ iff $w, i+1 \models \varphi$,
\item $w, i \models \varphi\,\mathcal{U}\,\psi$ iff there exists $k \geq i$ such that $w, k \models \psi$ and for all $i \leq j < k$, we have $w, j \models \varphi$.
\end{itemize}

We use the short hand notation  $\varphi \wedge \psi$, for $\neg (\neg \varphi \vee \neg \psi)$, and \textbf{True} for $\neg \varphi \vee \varphi$. We further abbreviate $\textbf{True}\, \mathcal{U}\, \varphi$ as $\Diamond \varphi$ which means that $\varphi$ \emph{eventually} holds and $\neg \Diamond \neg \varphi$ by $\square \varphi$, which says that $\varphi$ \emph{always} holds. We call the operators $\Circle$, $\mathcal{U}$, $\square$, and $\Diamond$ \emph{temporal operators}.

We write $W(\varphi)$ to denote the set of all infinite words which satisfy $\varphi$, i.e., $W(\varphi):=\{\sigma \in \Sigma^\omega | \sigma \models \varphi\}$. We say that $\psi_1$ and $\psi_2$ are \emph{semantically equivalent},  and write $\psi_1 \equiv~\psi_2$,
if $W(\psi_1)=W(\psi_2)$.

\end{subsection}

\begin{subsection}{Games}
\emph{A game graph} is a tuple $G=(V, E, P, L)$ consisting of:
\begin{itemize}\setlength\itemsep{0.4mm}
\item A finite set $V$ of states partitioned into $V_0$ and $V_1$, i.e., $V=V_0\cup V_1$ and $V_0\cap V_1=\varnothing$;
\item A \emph{transition relation} $E \subseteq V \times V$;
\item A finite set of atomic propositions $P$;
\item A \emph{labeling function} $L : V \to 2^P$ mapping every state in $V$ to the set of atomic propositions that hold true on that state.
\end{itemize}
In this definition, $V_0$ and $V_1$ are the states from which only player $0$ and player $1$ can move, respectively. Thus, the state determines which player can move. We assume that for every state $v \in V$, there exists some $v' \in V$ such that $(v,v') \in E$. The function $L$ can be naturally extended to infinite strings $r \in V^\omega$ by $L(r)=L(r_0)L(r_1)L(r_2)\ldots \in \Sigma^\omega$.

A \emph{play} $r$ in a game graph $G$ is an infinite
sequence of states $r=v_0v_1\ldots\in V^\omega$, such that for all $i \geq 0$,
we have $(v_i, v_{i+1}) \in E$. A strategy for player 0 is a partial function $f: V^* \times V_0 \to V$ such that
whenever $f(r,v)$ is defined $(v, f(r, v)) \in E$.
We denote the set of all plays under strategy $f$ starting from state $v$ by $\Omega_{f,v}(G)$, and the set of all possible plays for a given game graph $G$ by $\Omega(G)$. For a given LTL formula
$\varphi$ and a game graph $G=(V, E, P, L)$,
we use $W_G(\varphi)$ as the short-hand notation for $W(\varphi) \cap L(\Omega(G))$.

For the purposes of this paper, an \emph{LTL game} 
is a pair $(G, \varphi)$ consisting of a game graph $G$, and a \emph{winning condition}
$\varphi$ which is an LTL formula. 
A play $r$ in a game $(G,\varphi)$ is winning for player $0$ if $L(r) \in W(\varphi)$.
A strategy $f$ for player 0 is \emph{winning} from state $v$, if all plays starting in $v$ which follow $f$ are winning for player 0. For a given game $(G, \varphi)$, $\llbracket \varphi \rrbracket_G$
denotes the set of states from which player 0 has a winning strategy,
this is the \emph{winning set} of player 0. When it is clear from the context which game
graph we are referring to, we drop
the subscript and just write $\llbracket \varphi \rrbracket$.

The sets from 
which player 0 can force a visit to a set of states $V'$ is denoted by $\Pre(V')$, i.e.,
\begin{equation*}\Pre\left(V'\right)=\left\{v \in V_0\,\vert\, \exists_{v'\in V'}\,\, (v,v') \in E\right\}
 \cup \left\{v \in V_1 \,\vert\, \forall_{v'\in V}\,\, (v,v') \in E \Rightarrow v' \in V'\right\} \end{equation*}

We introduce the following fixed-point notation for a given monotone mapping $F: 2^V \to 2^V$:
\begin{align*}&\nu X F(X)=\cap_i X_i,\,\text{where}\, X_0=V,\, \text{and},\, X_{i+1}=F(X_i),\, \text{and}\\
&\mu X F(X)=\cup_i X_i,\,\text{where}\, X_0=\varnothing,\, \text{and}\, X_{i+1}=F(X_i).\end{align*}
In other words, $\nu X F(X)$ and $\mu X F(X)$ are the greatest and least fixed-point of the mapping $F$, respectively.

In the rest of the paper, we abuse notation and sometimes use a set of states $V' \subseteq V$
as an LTL formula. In this case $V'$ is to be interpreted as an atomic
proposition that holds only on the states in $V'$. Whenever, $V'$ defines an atomic proposition not in $P$, we can always extend $P$ to contain $V'$. However, for the sake of simplicity we will not explicitly do so.

We call $\varphi$ a \emph{positional formula} if it does not contain
any temporal operators and a \emph{reachability formula}
if $\varphi=\Diamond p$ for some positional formula $p$. We say that $\varphi$ is a \emph{GR(1) formula} if it has the following form:
\begin{equation}\label{Eqn:Gr1}\varphi=\bigwedge\limits_{i_1 \in I_1} \square \Diamond a_{i_1} \implies \bigwedge\limits_{i_2 \in I_2} \square \Diamond g_{i_2},\end{equation}
for some positional formulas $a_{i_1}$, $g_{i_2}$ and finite sets $I_1$ and $I_2$. We call games with winning conditions given as a GR(1) formula \emph{GR(1) games}. We refer the reader to \cite{GR1} for further details on GR(1) formulas.
\end{subsection}

\section{Mode-Target Games}\label{problemformulation}
\begin{subsection}{Motivation}As the automotive technology evolves, conventional cruise control (CCC) is being replaced by adaptive cruise control (ACC). ACC has two \emph{modes} of operation: the speed mode and the time-gap mode. In the speed mode, ACC behaves exactly like CCC, i.e., it reaches a pre-set speed and maintains it. The time-gap mode is what differentiates ACC from CCC. In this mode, ACC keeps pace with the car in front, the \emph{lead car}. This pace is characterized by the \emph{headway}, the quantity that captures the time required by the ACC equipped vehicle to break and avoid a collision when the lead car suddenly slows down. We consider the specifications for ACC set by the International Organization of Standardization (ISO) in \cite{ISOstandard}. Following these specifications, the target region corresponding to the speed mode can be defined as $v \in \{v : \left\vert v - v_{\text{des}}\right\vert \leq \epsilon_v\}$, where $v$, $v_{\text{des}}$ and $\epsilon_v$ denote the velocity of the car, the desired velocity, and the allowable tolerance for the velocity respectively. Similarly, the target region of the time gap mode is formalized as $\tau \in \{\tau : \left\vert \tau-\tau_{\text{des}}\right\vert \leq \epsilon_\tau\}$, where $\tau$ is the headway, $\tau_{\text{des}}$ is the desired headway, and $\epsilon_\tau$ is the desired tolerance for the headway.\footnote{In addition to the mode-target behavior, \cite{ISOstandard} requires the headway to be kept above a certain value regardless of the mode and at all times. However, this is a simple safety specification for which a controller can be synthesized separately and composed with the mode-target controller afterwards.} In each mode, the specification is to reach and stay in the desired target region as long as the current mode does not change. We can express this specification as the conjunction of individual specifications for the time-gap mode and the speed mode, i.e.,  $\varphi_{\text{timegap}} \wedge \varphi_{\text{speed}}$, where:
\begin{align}\label{MTformulaACC}
\varphi_{\text{timegap}}&:=\left(\Diamond \square M_{\text{timegap}} \implies \Diamond \square T_{\text{timegap}}\right),  \\
\label{MTformulaACC2}
\varphi_{\text{speed}}&:=\left( \Diamond \square M_{\text{speed}} \implies \Diamond \square T_{\text{speed}}\right).
\end{align}
Here, $M_{\text{timegap}}$ and $M_{\text{speed}}$ are the atomic propositions that hold whenever the corresponding modes are active. Similarly, $T_{\text{timegap}}$ and $T_{\text{speed}}$ are satisfied when  $\tau \in \{\tau : \left\vert \tau-\tau_{\text{des}}\right\vert \leq \epsilon_\tau\}$ and $v \in \{v : \left\vert v - v_{\text{des}}\right\vert \leq \epsilon_v\}$, respectively. 

Implication \eqref{MTformulaACC} only requires the time gap to be reached if  the system enters and stays in the time gap mode forever. Hence, it seems that a controller may simply ignore the time gap mode if it knows that this mode will be eventually left. However, since we synthesize \emph{causal} controllers, i.e., controllers that cannot foretell the future, any such controller will start driving the system to the time gap target once the system enters the time gap mode. Similarly, once the system leaves the time gap mode to enter the speed mode there is no need to reach the time gap mode anymore and the controller starts driving the system to the speed target. This is consistent with the ACC requirements in the ISO standard [12] that do not require a target to be reached once the corresponding mode is left.

We now consider an engine control example: the control of a combustion engine. As the researchers in the Toyota Technical Center argued in \cite{jyoHSCC}, the specifications for the air-fuel (A/F) ratio controller of an internal combustion engine can be naturally expressed in terms of modes and corresponding targets. We now summarize these specifications given in \cite{jyoHSCC}. There are four different modes of operation: start-up mode, normal mode, power-enrichment mode, and fault mode. Only one of these modes is active at any given time. Furthermore, for each mode there is a required A/F ratio. The specification for the controller is to bring the A/F ratio to this target value and keep it there unless the mode changes.  We compile the target A/F ratios corresponding to each mode in Table~\ref{AFratiotable}, where $\lambda_{\text{ref}}$,
and $\lambda_{\text{ref}}^{\text{pwr}}$ are the optimal A/F ratios for normal and ``full throttle"
driving conditions respectively.
\begin{table}
\caption {The modes and the corresponding target A/F ratios as given in \cite{jyoHSCC}. In this table, $\lambda_{ref}$, $\lambda_{ref}^{pwr}$ correspond to the optimal A/F ratios in normal and power-enrichment mode. The corresponding atomic propositions are written in parentheses.}
\begin{center}
    \begin{tabular}{ | l | l |}
    \hline
    \textbf{Mode} & \textbf{Target A/F Ratio}\\ \hline
    Start-up ($M_{\text{start-up}}$) & $[0.9\lambda_{ref}, 1.1\lambda_{\text{ref}}]\, (T_{\text{start-up}})$ \\ 
    Normal ($M_{\text{normal}}$) & $[0.98\lambda_{ref}, 1.02\lambda_{\text{ref}}]\,  (T_{\text{normal}})$ \\ 
    Power-Enrichment ($M_{\text{power}}$) & $[0.8\lambda^{\text{pwr}}_{\text{ref}}, 1.2\lambda^{\text{pwr}}_{\text{ref}}]\, (T_{\text{power}})$ \\
    Fault ($M_{\text{fault}}$) & $[0.9\lambda_{\text{ref}}, 1.1\lambda_{ref}]\,  (T_{\text{fault}})$  \\
    \hline
    \end{tabular}
\end{center}
\label{AFratiotable}
\end{table}
Defining the atomic propositions for modes and targets according to Table \ref{AFratiotable}, we get the following LTL formula that captures the desired behavior: $\varphi_{\text{start-up}} \wedge \varphi_{\text{normal}} \wedge \varphi_{\text{power}} \wedge \varphi_{\text{fault}}$, where
\begin{align*}\varphi_{\text{start-up}}&:=\left(\Diamond \square M_{\text{start-up}} \implies \Diamond \square T_{\text{start-up}}\right), \\
\varphi_{\text{normal}}&:= \left(\Diamond \square M_{\text{normal}} \implies \Diamond \square T_{\text{normal}}\right), \\
\varphi_{\text{power}}&:=\left(\Diamond \square M_{\text{power}} \implies \Diamond \square T_{\text{power}} \right), \\
\varphi_{\text{fault}}&:=\left(\Diamond \square M_{\text{fault}} \implies \Diamond \square T_{\text{fault}} \right).
\end{align*}

The last example we present is the control of a pressurized water reactor\footnote{A pressurized water reactor is a type of nuclear power plant that constitutes the majority of  nuclear power plants in Western countries, including the US.} during shutdown and start-up stages. Even though the chemical processes that take place in nuclear power plants are well studied under normal conditions, they are still yet to be fully understood in the presence of transient behaviors, particularly during shutdown and start-up.
Therefore, it is important to ensure correct operation during these critical phases. In \cite{nuclearpowerplant}, the authors document the specifications set by \'Electricit\'e de France (EdF) for both of these modes of operation. Here we present a simplified version of these specifications. According to \cite{nuclearpowerplant}, there are two shutdown procedures that can be followed based on the current temperature and concentration of the materials in the plant: hot shutdown and cold shutdown. In the hot shutdown mode, there is a target hydrogen concentration that must be achieved. In the cold shutdown mode, the shutdown can be performed with or without oxygenation depending on factors such as financial cost, risk, and specifics of the power plant. For both of these modes the control objective is to attain and sustain a certain chemical content in the reactor. Table \ref{nuclearpowerplanttable} summarizes these target chemical concentrations corresponding to each operation mode.
\begin{table}
\caption {The modes and the targeted concentration of chemicals in each mode as given in \cite{nuclearpowerplant}. In parentheses, we provide the notation for the atomic propositions corresponding to each mode and target.}
\begin{center}
    \begin{tabular}{ | c | c |}
    \hline
    \textbf{Mode} & \textbf{Target Chemical Content}\\ \hline
    Start-up ($M_{\text{start-up}}$)&Sodium $<$ 0.1 mg/kg \\ &Hydrazine $>$ 0.1 mg/kg ($T_{\text{start-up}}$) 
\\ \rule{0pt}{4ex}    
   Hot shutdown ($M_{\text{hot}}$)  & $15 \text{cm}^3/\text{kg}< \hspace{-0.7mm}\text{H}_2 \hspace{-0.7mm}< 50\text{cm}^3/\text{kg}$\,\,($T_{\text{hot}}$) \\
   \rule{0pt}{4ex}    
    Cold shutdown ($M_{\text{cold}}$) &    $\text{O}_2 > 1\, \text{mg}/\text{kg}$\, ($T_{\text{cold, w / oxy}}$) \hspace{-2mm} \\
   & $\text{H}_2 > 50\, \text{N}\,\text{cm}^3\text{kg}$\, ($T_{\text{cold, w/o oxy}}$) \\
    \hline
    \end{tabular}
\label{nuclearpowerplanttable}
\end{center}
\end{table}
Accordingly, in this case the LTL formula describing the desired behavior is $\varphi_{\text{start-up}} \wedge \varphi_{\text{cold}} \wedge \varphi_{\text{hot}}$, which is conjunction of the specifications for the start-up mode, the hot shutdown mode, and the cold shutdown mode, where
\begin{align*}\varphi_{\text{start-up}}&:=\left( \Diamond \square M_{\text{start-up}} \implies \Diamond \square T_{\text{start-up}}\right), \\
\varphi_{\text{hot}}&:= \left( \Diamond \square M_{\text{hot}} \implies \Diamond \square T_{\text{hot}}\right), \\
\varphi_{\text{cold}}&:= \left( \Diamond \square M_{\text{cold}} \implies \left(\Diamond \square T_{\text{cold, w/ oxy}} \vee \Diamond \square T_{\text{cold, w/o oxy}}\right)\right).
\end{align*}
\end{subsection}
\vspace{-3mm}
\begin{subsection}{Mode-Target Formulas and Games}
The preceding examples illustrate the  scenarios that we want to
capture with a suitable LTL fragment. All of the control problems we just described share
the following properties that define our setting:\\
\textbf{(P1)} There are modes and corresponding targets.\\
\textbf{(P2)} If the system enters a mode, it should reach one of the targets associated with that mode and remain there. \\
\textbf{(P3)} If the mode changes, there is no obligation to reach any of the targets of the previous mode anymore.

We also make the following observation regarding the dynamics of the modes:\\
\textbf{(P4)} There is at most one mode active at any given time.

With these properties in mind, we now formally define mode-target formulas and games. For a game to be a mode-target game, its winning condition must be given by a mode-target formula and the corresponding game graph should have a specific structure capturing \textbf{(P1)}-\textbf{(P4)}.

Let $T$ and $M$ be finite sets of atomic propositions:
$T=\cup_i T_i$ and \linebreak $M=\left\{M_1, M_2,\ \ldots M_{m}\right\}$, where \mbox{$T_i=\{T_{i,1},T_{i,2}, \ldots, T_{i,t_i}\}$.} Here, the $M_i$, $T_{i,j}$ represent the mode $i$, and $j^{th}$ target of mode $i$ respectively. We start with a game graph $G$ labeled with modes and targets, i.e., \mbox{$G=(V, E, M \cup T, L)$} where \mbox{$L: V \to 2^{M \cup T}$.} The winning condition for player 0 is given by a mode-target formula.
\begin{definition}[Mode-Target Formula] An LTL formula is a \emph{mode-target formula} if it has the form
\begin{equation}\label{fragment_old}\varphi:=\bigwedge\limits_{i=1}^m \left(\Diamond \square M_i \implies \bigvee\limits_{j=1}^{t_i}\Diamond \square T_{i,j}\right).\end{equation}
\end{definition}

We can interpret $\varphi$ as: \emph{if the system eventually settles in
$M_i$, then it should eventually settle in one of the modes in $T_i$.} This formula captures \textbf{(P2)}
because it guarantees that the system will reach one of the target regions in $T_i$
if the system stays in mode $M_i$ from a certain time onwards. As we explained previously, the left-hand side of the implication in \eqref{fragment_old} ensures that if the mode changes, the system does not have to reach or stay in any of the corresponding targets of the previous mode, as asserted by \textbf{(P3)}. It is true that $\varphi$ can also be satisfied
by switching between modes infinitely often. However, as it is the case in the
ACC, A/F ratio, and pressurized water reactor examples, the modes can be partially if not fully
determined by an external signal that the controller cannot change.
In these cases, by construction, the controller will make progress
towards the target of the current mode since it cannot predict if the system will remain in the current mode or switch to a different mode. Also note that for the ACC and A/F ratio control examples each $T_i$ is simply a singleton, since there is only one target region that can be reached for all modes. This is not the case, however, for the pressurized water reactor control example.

To address \textbf{(P4)} we make the following assumption on the modes:\\
\textbf{(A)} Modes are mutually exclusive, i.e., $M_i \in L(v)\implies 
M_j \notin L(v), \, \forall j \not=i,\forall v \in V. $

\begin{definition}[Mode-Target Games]
We call LTL games with winning condition given by a mode-target formula and a labeling function $L$ that satisfies
\textbf{(A)}, \emph{mode-target games}.
\end{definition}
\end{subsection}

Note that, a mode-target game is a Streett game \cite{Streett} with additional structure imposed by the assumption \textbf{(A)} on the labeling function.

\section{Solving Mode-Target Games}
\label{directMT}
\subsection{Decomposition of the Winning Set}
We start by introducing a few notions that are critical to
understand the solution of MT games described in this section.

Let $S_1 \subseteq \Sigma^*$ and $S_2 \subseteq \Sigma^* \cup \Sigma^\omega$. We define the concatenation of these sets as 
\begin{equation*}S_1 S_2:=\{\sigma \in \Sigma^* \cup \Sigma^\omega | \sigma=\sigma_1 \sigma_2, \, \sigma_1 \in S_1,\,\sigma_2 \in S_2\}.\end{equation*} A \emph{property} $\Phi$ is a subset of $\Sigma^\omega$.
The set of suffixes of a property $\Phi$ is denoted by $\Post(\Phi)$, i.e.,
$\Post(\Phi):=\left\{\sigma' \in \Sigma^\omega | \sigma\sigma' \in \Phi,\, \text{for some $\sigma \in \Sigma^*$}\right\}.$
A property $\Phi$ is an \emph{absolute liveness property} iff
$\Sigma^*\Phi \subseteq \Phi$. We call $\varphi$ an \emph{absolute liveness formula}
if $W(\varphi)$ is an absolute liveness property. A formula $\varphi$ is an absolute liveness
formula iff $\varphi \equiv \Diamond \varphi$ (see \cite{Sistla}).
It follows that any formula of the form $\Diamond \phi$, for some $\phi$ is an absolute liveness formula. 

We now introduce a class of games that includes both GR(1) games and MT games. The definition of this class of games distills the properties that are essential for a simple and transparent derivation of its solution.

\begin{definition}
An LTL game $(G,\varphi)$ is said to be \emph{simple} if the winning condition defined by $\varphi$ can be written as:
\begin{equation}
\label{Eq:WinningCondition}
\mbox{$\varphi=\square \bigwedge\limits_{i \in I} \varphi_i$}, \qquad \varphi_i=\Diamond p_i \vee \psi_i,
\end{equation}
where $p_i$ is a positional formula and $\psi_i$ is an absolute liveness formula that satisfies:
\begin{equation}
\label{Eq:Inclusion}
W_G(\psi_i) \subseteq W(\varphi).
\end{equation}
\end{definition}

\begin{lemma}
\label{Lemma:GR1toMT}Every GR(1) game is a simple game. 
\end{lemma}

\begin{proof} See Appendix \ref{appendix3}. 
\end{proof}

The proof of Lemma \ref{Lemma:GR1toMT} relies on showing that any GR(1) formula can be written in the following form:
\begin{equation}\label{GR1simple}\square\bigwedge\limits_{i_2 \in I_2} \left(\Diamond g_{i_2} \vee \left(\bigvee\limits_{i_1 \in I_1} \Diamond \square \neg a_{i_1}\right) \right).
\end{equation}
The formula in \eqref{GR1simple} satisfies the properties required by the winning condition of simple games given in \eqref{Eq:WinningCondition} and \eqref{Eq:Inclusion}, where $g_{i_2}$ is the positional formula $p_i$ and $\left(\bigvee\limits_{i_1 \in I_1} \Diamond \square \neg a_{i_1}\right)$ is $\psi_i$. The inclusion in \eqref{Eq:Inclusion} is also fulfilled since we have
\begin{equation*}
W\left(\bigvee\limits_{i_1 \in I_1} \Diamond \square \neg a_{i_1}\right) \subseteq W\left(\left( \square\bigwedge\limits_{i_2 \in I_2} \left(\Diamond g_{i_2} \vee \left(\bigvee\limits_{i_1 \in I_1} \Diamond \square \neg a_{i_1}\right) \right) \right)\right).
\end{equation*}

\begin{lemma}\label{lemma:mdfragment}Every mode-target game is simple.
\end{lemma}

\begin{proof}See Appendix \ref{app2proof1}. 
\end{proof}

We prove Lemma \ref{lemma:mdfragment} by showing that every MT formula can be written as:
\begin{equation}\label{MTsimple}\square \bigwedge \limits_{i=1}^m\left(\Diamond \neg M_i \vee \left(\bigvee\limits_{j=1}^{t_i}\Diamond \square (M_i \wedge T_{i,j})\right)\right).
\end{equation}
Note that \eqref{MTsimple} is in the form defined by \eqref{Eq:WinningCondition} and \eqref{Eq:Inclusion}, where the positional formula $p_i$ is $\neg M_i$ and formula $\psi_i$ is $\bigvee\limits_{j=1}^{t_i}\Diamond \square (M_i \wedge T_{i,j})$.


The winning condition for simple games can be written as a conjunction of formulas $\varphi_i$ preceded by $\square$ where each $\varphi_i$ can be decomposed as a disjunction between a reachability formula and a formula $\psi$ satisfying (\ref{Eq:Inclusion}). We now show that it is easy to modify algorithms that synthesize winning strategies for reachability games to obtain an algorithm for a conjunction of reachability formulas preceded by $\square$. The approach in this algorithm remains valid even when we disjoin these reachability formulas with absolute liveness formulas $\psi_i$'s, in virtue of~(\ref{Eq:Inclusion}). The inclusion given in~(\ref{Eq:Inclusion}) ensures that a play in $(G, \psi_i)$ that is winning for player $0$, is also winning in $(G, \varphi)$. Therefore, one can adopt a compositional approach to the solution of simple games. A small modification to an algorithm that computes $\llbracket \varphi_i \rrbracket$ leads to an algorithm computing $\llbracket \square \bigwedge_{i\in I}\varphi_i \rrbracket$. The next result makes these ideas precise.

\begin{theorem}\label{theorem:compsound}The winning set for player $0$ in a simple game $(G,\varphi)$ is given by
\begin{equation}
\label{Eq:SimpleSolution}
\left\llbracket \varphi \right\rrbracket = \nu Z \bigcap\limits_{i \in I} \left\llbracket\psi_i \vee \Diamond (p_i \wedge \Circle Z) \right\rrbracket.
\end{equation}
\end{theorem}

\begin{proof}See Appendix \ref{soundness_app}.
\end{proof}

The proof of the first part of Theorem \ref{theorem:compsound}, follows the existing methods for constructing winning strategies for Generalized B{\"u}chi games \cite{GeneralizedBuchi}, in which the winning condition is given by
\begin{equation}
\label{GeneralizedBuchi}
\bigwedge\limits_{i \in I} \square \Diamond B_i \equiv \square \bigwedge_{i \in I}\Diamond B_i,
\end{equation}
for some subset of states $B_i \subseteq V$. The winning condition we are interested in, given in~(\ref{Eq:WinningCondition}), is slightly different from the one given in \eqref{GeneralizedBuchi} due to the additional $\psi_i$ term. However, inclusion \eqref{Eq:Inclusion} ensures that any play that is winning for $(G, \psi_i)$ is also winning for $(G, \varphi)$. Hence, by simply computing \mbox{$\nu Z \bigcap_{i \in I} \left\llbracket \psi_i \vee \Diamond (p_i \land \Circle Z)\right\rrbracket$} we can obtain a winning strategy for player $0$ in a simple game. Moreover, this strategy can be seen as the composition of the strategies for games with the simpler winning condition $\psi_i \vee \Diamond (p_i \land \Circle Z)$.

Theorem~\ref{theorem:compsound} shows how the structure of simple games makes it possible to combine the sets $\llbracket \psi_i \vee \Diamond p_i \rrbracket$ as in \eqref{Eq:SimpleSolution} to compute the final winning set. In particular, we conclude that modularity observed in the solution of GR(1) games is not due to the structure of GR(1) formulas but rather to the structure of simple game formulas. Hence, this structure can be leveraged beyond GR(1) games as we did for MT games. Note how Theorem~\ref{theorem:compsound} describes the solution to both GR(1) and MT games. For later reference we instantiate~(\ref{Eq:SimpleSolution}) for MT games:
\begin{equation}
\label{Eq:MTSolution}
\left\llbracket \varphi \right\rrbracket =\nu Z \bigcap\limits_{i=1}^m \left\llbracket \bigvee\limits_{j=1}^{t_i} \Diamond \square (M_i \wedge T_{i,j}) \vee \Diamond (\neg M_i \wedge \Circle Z) \right\rrbracket
\end{equation}
and explain in the next section how to compute the winning sets
\begin{equation}\label{winningsetMT}\left\llbracket \bigvee\limits_{j=1}^{t_i} \Diamond \square (M_i \wedge T_{i,j}) \vee \Diamond (\neg M_i \wedge \Circle Z) \right\rrbracket\end{equation}
so as to make use of~(\ref{Eq:MTSolution}). Note that if we instead instantiate~(\ref{Eq:SimpleSolution}) for the GR(1) formula~(\ref{Eqn:Gr1}) we obtain
\begin{equation}
\label{winningGR1}
\nu Z \bigcap\limits_{i_2 \in I_2} \left\llbracket \bigvee\limits_{i_1 \in I_1} \Diamond \square \neg a_{i_1} \vee \Diamond (\neg g_{i_2} \wedge \Circle Z) \right\rrbracket.
\end{equation}

The structures of the fixed-point expressions given in \eqref{winningGR1} and \eqref{Eq:MTSolution} are very much alike, but not the same. While 
in GR(1) games for each $i_2 \in I_2$, i.e., for each guarantee, the same persistency property is required to be satisfied ($\lor_{i_i\in I_1}\Diamond\Box\neg a_{i_1}$), in the case of MT games, the persistency part of the specification depends on the current mode, i.e., the index $i$, as in \eqref{Eq:MTSolution} ($\lor_{j=1}^{t_i} \Diamond \square (M_i \wedge T_{i,j})$).


\subsection{Computation of the Winning Set}
\label{compWinningSet}
In \cite{Kesten}, Kesten, Piterman and Pnueli presented a $\mu$-calculus formula which characterizes $\left\llbracket \vee_{i \in I} \Diamond \square p_i \vee \Diamond q\right\rrbracket$, where $p_i$ and $q$ are positional formulas.
This $\mu$-calculus formula yields the following fixed-point expression:
\begin{equation}\label{Eq:KestenSln}\mu Y \bigcup\limits_{i \in I} \left(\nu X( \Pre(X) \cap \llbracket p_i \rrbracket) \cup \llbracket q \rrbracket \cup \Pre(Y)\right).\end{equation}
Using \eqref{Eq:KestenSln} it is easy to see that the winning set~(\ref{Eq:MTSolution}) is given by the following fixed-point:
\begin{align}\label{finalfixed-point}
\llbracket\varphi\rrbracket=\nu Z \left(\bigcap\limits_{i =1}^m \mu Y \bigcup\limits_{j=1}^{t_i}\left(\nu X (\Pre(X) \cap \llbracket M_i \land T_{i,j} \rrbracket\right)\right.
 \left.\cup \vphantom{\bigcap\limits_{i=1}^m}\hspace{0mm}\left(\llbracket \neg M_i \rrbracket \cap \Pre(Z)) \cup \Pre(Y)\right)\right).
\end{align}

We refer to the algorithm defined by the iterative computation of the preceding fixed-point as the \textbf{MT} algorithm. In the worst case, the \textbf{MT} Algorithm can take $O(\sum_i t_i n^2 )$ iterations, where $t_i$ is the number of targets dedicated to mode $i$ and $n$ is the number of vertices in the game graph $G$. We summarize this in the following theorem.

\begin{theorem}\label{timecomplexity2}Mode-target games can be solved by the symbolic algorithm \textbf{MT} requiring $O(\sum_{i=1}^m t_i\, n^2)$ $\Pre$ computations.
\end{theorem}

\begin{proof}In \cite{Browne} Browne et al. show that a fixed point expression with \emph{alternation depth} $k$ can be computed in $O(n^{\lfloor1+k/2\rfloor})$ iterations. Note that given a fixed-point expression the alternation depth is simply the number of alternating greatest and least fixed point operators.

The alternation depth of the fixed-point expression \eqref{finalfixed-point} is three. Moreover, the computation of the fixed-point involves sequentially evaluating $t_i$ fixed-point expressions for each mode, which results in $O\left(\sum_{i=1}^m t_in^2\right)$ $\Pre$ computations in the worst case.
\end{proof}

Theorem \eqref{timecomplexity2} only addresses the computation of the winning set for the controller. However, the fixed-point computation given in \eqref{finalfixed-point} is constructive in the sense that we can find a winning strategy by storing the intermediate sets that are computed during its evaluation. The precise construction and implementation of the winning strategy follows the same approach as in GR(1) games \cite{GR1}. For the sake of completeness we provide the details of the winning strategy synthesis in Appendix \ref{synthesis}. Note that contrary to the winning strategy for GR(1) games, the winning strategy for MT games is memoryless since player 0 only needs to know what the current mode is.

\section{Solving Mode-Target Games via GR(1) Games}\label{algorithmsection}
In this section, we describe how to transform a given MT game into a GR(1) game, thereby obtaining another algorithm to solve MT games that is based on the existing synthesis algorithms for the GR(1) fragment. To simplify the notation in the next proposition we introduce the atomic proposition $\bar{T}_{i,j}$ defined by:
\begin{equation*}
\bar{T}_{i,j} =
\begin{cases}
T_{i,j},& \text{if $j \leq t_i$}\\
\mathtt{false} & \text{otherwise}.
\end{cases}
\end{equation*}
\begin{proposition}\label{prop_embed1}Every MT game with game graph $G$ is equivalent to the GR(1) game $(G,\varphi)$, where
\begin{equation}
\label{Eq:Emb1}
\varphi=\left(\bigwedge\limits_{j=1}^{\max_i{t_i}} \square  \Diamond \wedge_{i=1}^m (\neg M_i \vee \neg \bar{T}_{i,j})\hspace{-1mm}\right) \hspace{-2mm}\implies\hspace{-2mm} \left(\bigwedge\limits_{i=1}^m \square \Diamond \neg M_i\right)\hspace{-1mm},
\end{equation}
\end{proposition}

\begin{proof} See Appendix \ref{embedding1}.
\end{proof}

The proof of \eqref{Eq:Emb1} has two main steps. In the first step, we show that the MT game is equivalent to the GR(1) game $(G, \varphi_1)$, where
\begin{equation}\label{firstGR1}\varphi_1=\left(\bigvee\limits_{i=1}^{m}\bigvee\limits_{j=1}^{\max_i t_i} \Diamond \square (M_i \wedge \bar{T}_{i,j}) \right) \vee \left(\bigwedge\limits_{i=1}^m \square \Diamond \neg M_i\right).
\end{equation}
The equivalence of $(G, \varphi)$ to the MT game relies on assumption \textbf{(A)}. Also note that the formula in \eqref{firstGR1} is satisfied either when the system settles down in a mode and in one of the corresponding targets or when it toggles between the modes indefinitely, which matches the initial motivation of the MT fragment. Since the formula given in \eqref{firstGR1} is a GR(1) formula with $\sum_i t_i$ assumptions and $m$ guarantees, this part of the proof already leads to a synthesis algorithm for MT games. In the second part of the proof we show\footnote{This part of the proof is based on a comment we received from an anonymous reviewer of the preliminary version of our results presented in \cite{ADHSBalkan}.} how to construct a GR(1) game with fewer assumptions that is equivalent to $(G,\varphi_1)$ and for which the statement of Proposition \ref{prop_embed1} holds. Again assumption \textbf{(A)} lies at the heart of the proof. This assumption restricts the modes to be mutually exclusive and therefore enforces additional structure on MT games, which lets us simplify the formula in \eqref{firstGR1}.

The formula given in \eqref{Eq:Emb1} is a GR(1) formula with $\max_i t_i$
assumptions, and $m$ guarantees. Notice that this formula has at most the same number of assumptions as $\varphi_1$ since $m \max_i t_i \leq m \sum_i t_i$. Due to Proposition \ref{prop_embed1} we can now simply apply the algorithm given in \cite{GR1} to the game graph $G$ with the winning condition \eqref{Eq:Emb1} to solve the MT game. This algorithm is based on the computation of the following fixed-point:
\begin{align}\label{Fixpoint:Emb1}
\nu Z\hspace{-0.6mm} \left(\bigcap\limits_{i=1}^m\hspace{-0.6mm}  \mu Y \hspace{-0.9mm}  \left(\bigcup\limits_{j=1}^{\max_i{t_i}}\hspace{-0.6mm} \nu X\hspace{-0.6mm} \left(\Pre(X) \cap\hspace{-0.2mm} \cup_{\ell=1}^m \llbracket M_\ell \wedge \bar{T}_{\ell,j}\rrbracket\right)\right.\right. \left. \left.\phantom{\bigcup\limits_j^k}\hspace{-6.7mm}  \cup \hspace{-0.2mm}\Pre(Y) \cup \hspace{-0.5mm} \left(\llbracket \neg M_i \rrbracket \cap \Pre(Z)\right) \right)\hspace{-1mm}\right)\hspace{-0.7mm}.
\end{align}

We refer to the algorithm defined by the iterative computation of the preceding fixed-point as the \textbf{GR(1)-Emb} algorithm for \emph{GR(1) Embedding}. In the worst case, the \textbf{GR(1)-Emb} algorithm can take $O(m\max_i t_in^2)$ iterations, where $m$ is the number of modes in the MT formula, $t_i$ is the number of targets dedicated to mode $i$, and $n$ is the number of vertices in the game graph $G$. This follows from the fact that solving GR(1) games according to the fixed-point computation in~\cite{GR1} takes $O(n_an_gn^2)$ symbolic steps where $n_g$ is the number of guarantees and $n_a$ is the number of assumptions. Then, the bound $O(m\max_i t_i n^2)$ follows from the fact that $n_a=\max t_i$ and $n_g=m$ as in \eqref{Eq:Emb1}.

The following result summarizes the discussion in this section.

\begin{theorem}\label{timecomplexity}Mode-target games can be solved by the symbolic algorithm \textbf{GR(1)-Emb} requiring $O(m\max_i t_i n^2)$ $\Pre$ computations. 
\end{theorem}

\begin{proof}Similar to the proof of Theorem \ref{timecomplexity2}, this result follows from the fact that the given fixed-point expression is of alternation depth three. Moreover, in each iteration of the algorithm we sequentially compute $m \max_i t_i$ fixed-point expressions which results in $O(m\max_i t_i n^2)$ $\Pre$ computations in the worst case.
\end{proof}

Comparing the complexities of the \textbf{MT} and the \textbf{GR(1)-Emb} algorithms as given in Theorem \ref{timecomplexity} and Theorem \ref{timecomplexity2}, we get
\begin{equation}\label{Eq:ComparisonMTGR1}O\left(\sum_{i=1}^m t_i\, n^2\right) \leq O\left(m \max_i t_i\, n^2\right).
\end{equation}
Although the \textbf{GR(1)-Emb} and the \textbf{MT} algorithms compute the same winning set, the \textbf{MT} algorithm has better worst case complexity than the \textbf{GR(1)-Emb} algorithm.
Moreover, the equality in \eqref{Eq:ComparisonMTGR1} holds iff
\begin{equation}t_\ell= \max_i t_i\, \qquad \text{for all}\, \ell \in \{1,2, \ldots m\},
\end{equation}
i.e., if the number of targets associated with each mode is equal. In this special case, assuming the number of targets for each mode to be $t$, the fixed-point that needs to be computed for the \textbf{GR(1)-Emb} algorithm is
\begin{align}\label{GR1sametarget}
\nu Z \left(\bigcap\limits_{i=1}^m \mu Y \bigcup\limits_{j=1}^{t}\left(\nu X (\Pre(X) \cap \cup_{\ell=1}^m \llbracket M_\ell \land T_{\ell,j} \rrbracket\right)\right. 
\left.\cup\, \hspace{1mm}\vphantom{\bigcap\limits_{i=1}^m}\hspace{-2mm}\left(\llbracket \neg M_i \rrbracket \cap \Pre(Z)) \cup \Pre(Y)\right)\right),
\end{align}
while for the \textbf{MT} algorithm the fixed-point computation given in \eqref{finalfixed-point} becomes
\begin{align}\label{finalfixed-pointsametarget}
\nu Z \left(\bigcap\limits_{i =1}^m \mu Y \bigcup\limits_{j=1}^{t}\left(\nu X (\Pre(X) \cap \llbracket M_i \land T_{i,j} \rrbracket\right)\right. 
\left.\cup\, \vphantom{\bigcap\limits_{i=1}^m}\hspace{0mm}\left(\llbracket \neg M_i \rrbracket \cap \Pre(Z)) \cup \Pre(Y)\right)\right).
\end{align}

As can be seen from \eqref{GR1sametarget} and \eqref{finalfixed-pointsametarget}, even in this special case where the two different approaches have the same worst-case complexity, the computations performed by \textbf{GR(1)-Emb} and \textbf{MT} differ. While the fixed-point expression \eqref{finalfixed-pointsametarget} has $\cup_{\ell=1}^m \llbracket M_\ell \wedge T_{\ell,j} \rrbracket$ for every mode index $i$, the fixed-point in \eqref{finalfixed-pointsametarget} replaces this set with $\llbracket M_i \wedge T_{i,j} \rrbracket$ for each $i$. Since  $\llbracket M_i \wedge T_i \rrbracket \subseteq \cup_{\ell=1}^m \llbracket M_\ell \wedge T_{\ell,j} \rrbracket$ for all $i$ and $j$, due to the monotonicity of the given fixed-point operator, the \textbf{MT} algorithm performs no worse than the \textbf{GR(1)-Emb} in terms of number of iterations. Moreover, for a given $i$ and $j$, in order to compute the fixed-point in the variable $X$, the algorithm \textbf{MT} only requires the storage of the set $\llbracket M_i \wedge T_{i,j} \rrbracket$ instead of $\cup_{\ell=1}^m \llbracket M_\ell \wedge T_{\ell,j} \rrbracket$. This suggests that the algorithm \textbf{MT} might also have better space complexity. To investigate these differences in practice, we provide in the next section an experimental comparison of two implementations for each of the two algorithms presented in this paper: \textbf{GR(1)-Emb}, and \textbf{MT}.

\section{Experimental Comparison}\label{experiments}
The winning set and a corresponding winning strategy can be computed by iterating the operators on the right hand sides of  \eqref{finalfixed-point} and \eqref{Fixpoint:Emb1} until a fixed-point is reached.  We can improve the time efficiency of a direct implementation of this iteration by using two important ideas from the literature. In \cite{emersonlei}, the authors make the following observation: if one wants to compute the largest (smallest) fixed-point of an operator and one already knows a set that contains (is contained in) this fixed-point, then the largest (smallest) fixed-point computation can be started from this value instead of $V$ ($\varnothing$). By using this idea, the authors showed that the complexity of their computation does not depend on the number of fixed-point operators but rather the number of such fixed-point alternations, i.e., alternation depth. Taking the same idea a step further, in \cite{Browne}, by exploiting monotonicity, the authors point state that one can use the intermediate values of the sets to initialize the fixed-point computations. This method also leads to improved time efficiency, but now with the cost of the requirement to store the value of intermediate sets that are not necessary for the computation of the final fixed-point. However, as mentioned in Section \ref{compWinningSet}, the construction of the winning strategy depends upon these intermediate values. Therefore, in our experiments we use the method described in \cite{Browne}, since the extra memory allocation is partly unavoidable when the desired end product is a winning strategy, and not just the winning set. 

In this section, we discuss the experimental time and memory usage of algorithms \textbf{GR(1)-Emb} and \textbf{MT}. We present three sets of experiments. The first two are designed to compare the performance of the two algorithms in different scenarios, while the last one demonstrates a concrete application of the MT fragment in the design of the ACC example described in Section \ref{ACCexample}.

\begin{subsection}{Random Linear Time-Invariant Systems with Multiple Targets}\label{exampleLTI}
We start with the simplest class of dynamical systems: linear time-invariant systems. We demonstrate how the performance of the two algorithms differs as the theoretical worst-case gap between the \textbf{GR(1)-Emb} algorithm and \textbf{MT} algorithm deepens. To this end, we consider a scenario where all modes but one have a single associated target. For this remaining mode, starting from a single target we gradually increase the number of associated targets in order to accentuate the difference between the two sides of the inequality in \eqref{Eq:ComparisonMTGR1}. We provide the descriptions of all mode and target sets
in Appendix \ref{exampleData}. In Fig. \ref{Fig:multipleTargets}, we summarize our findings for the case when we have three, six and nine modes. We plot in Fig. \ref{fig:mulitpleTargets1} the ratio between the number of iterations it takes for the \textbf{GR(1)-Emb} algorithm versus the \textbf{MT} algorithm to compute the winning set. In Fig. \ref{fig:multipleTargets2} we compare the two algorithms in the same fashion, but now in terms of the elapsed time. Each data point represents the average value we obtained after computing the winning set on 20 random linear time-invariant systems. All systems have the form $\dot{x}=Ax+Bu$, where the entries of the matrices $A$ and $B$ are randomly chosen from the set $[-1,1]$. The state space and the input space are the sets $[-6,6] \times [-6, 6]$, and $[-4, 4]$, respectively. As can be seen from both figures,
 \textbf{MT} outperforms \textbf{GR(1)-Emb}, and the performance difference becomes progressively more prominent as the number of extra targets and modes increase. 

\begin{figure}[h]
\centering
\begin{subfigure}[h]{0.485\textwidth}
\centering
   \includegraphics[scale=0.365,clip=true,trim=66 190 55 195]{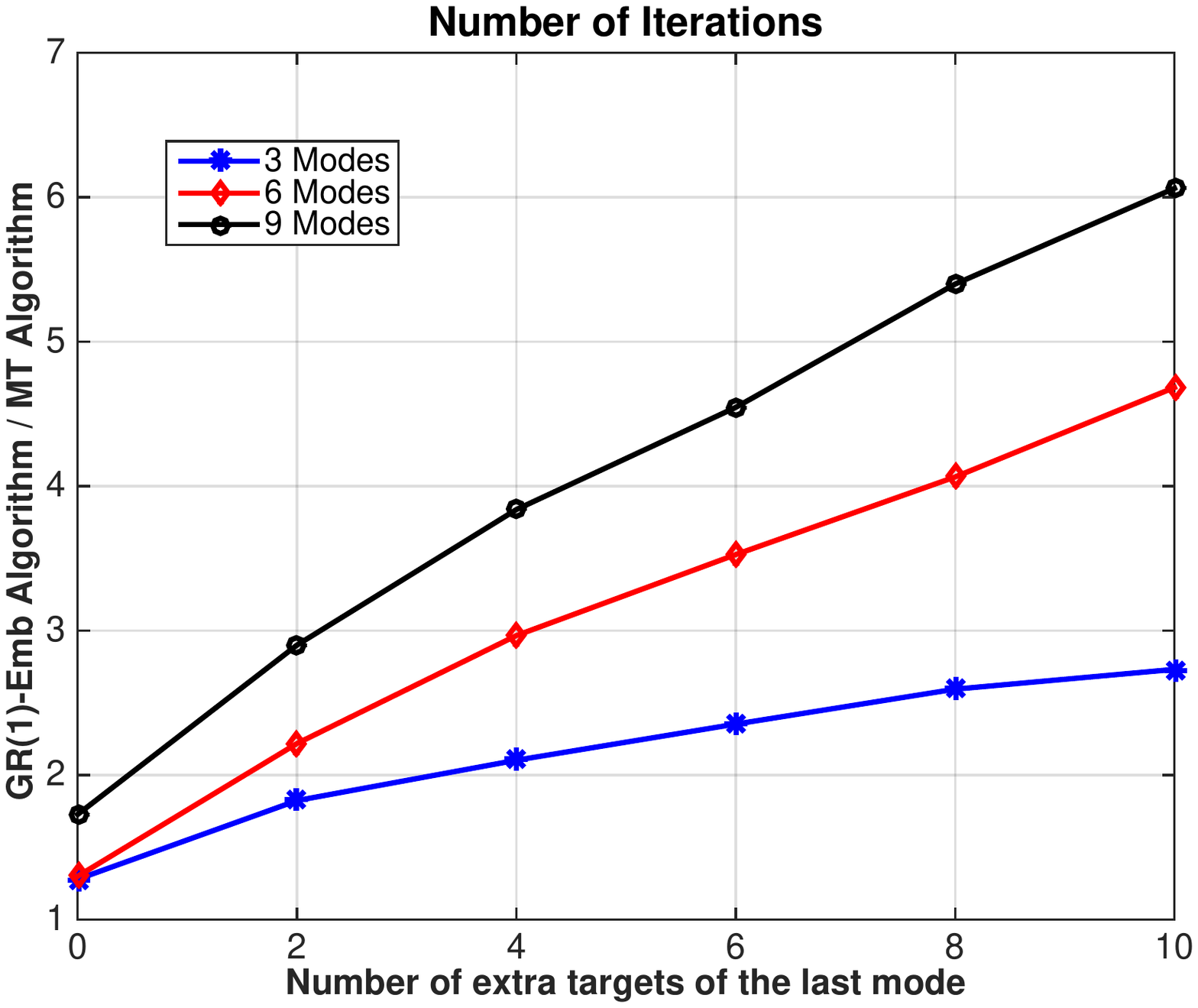}
   \caption{The ratio of number of iterations of \textbf{GR(1)-Emb} to \textbf{MT}.}
   \label{fig:mulitpleTargets1}
\,
 \end{subfigure}
 \hspace{1.4mm}
\begin{subfigure}[h]{0.485\textwidth}
\centering
   \includegraphics[scale=0.365,clip=true,trim=67 190 55 195]{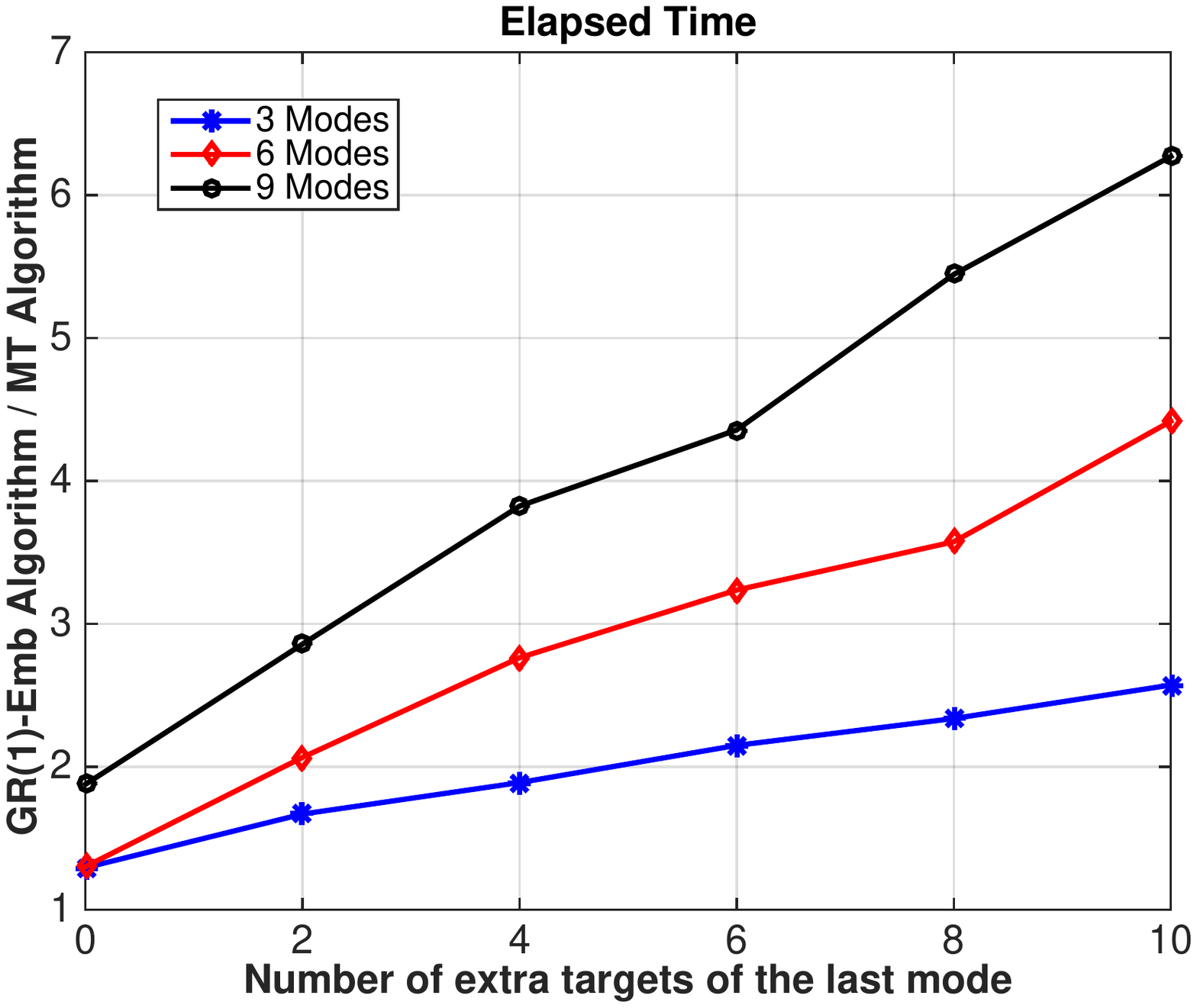}
   \caption{The ratio of elapsed time until convergence of \textbf{GR(1)-Emb} to \textbf{MT}.}
   \label{fig:multipleTargets2}
 \end{subfigure}\\
 \caption{Comparison of the algorithms \textbf{GR(1)-Emb} and \textbf{MT} when there are multiple targets corresponding to one of the modes.}
 \label{Fig:multipleTargets}
\end{figure}
\end{subsection}

\begin{subsection}{Unicycle Cleaning Robot}
We consider a scenario where a unicycle robot cleans the rooms on a hotel floor. The robot has to reach one of the rooms that is not clean and stay there, until an external signal indicates that the current room has been cleaned. We now explain how we model this scenario as an MT game. Assume that there are two rooms, defined by the atomic propositions $T_1$ and $T_2$. Each mode-target pair corresponds to a different subset of rooms  that need to be cleaned. Specifically, $M_1$, $M_2$, and $M_3$ indicate that only the first room, only the second room, and both of the rooms need to be cleaned, respectively. Accordingly, the MT formula corresponding to this scenario is:
\begin{equation*}(\Diamond \square M_1 \implies \Diamond \square T_1) \wedge (\Diamond \square M_2 \implies \Diamond \square T_2) \wedge (\Diamond \square M_3 \implies (\Diamond \square T_1 \vee  \Diamond \square T_2)).
\end{equation*}
Note that, if there are $k$ rooms, the number of modes is $2^k-1$.

We first construct the game graph corresponding to the dynamics of the cleaning robot. The differential equations:
\begin{equation*}
\dot{x}=v \cos(\theta),\,\dot{y}=v \sin(\theta),\,\dot{\theta}=\omega,
\end{equation*}
offer a simplified model for a 3-wheel robot equipped with differential drive. The pair $(x,y) \in \R^2$ denotes the position of the robot, $\theta \in [-\pi,\pi[$ denotes its orientation, and $(v,\omega) \in \R^2$ are the control inputs, linear velocity $v$ and angular velocity $\omega$. For this example we restrict the position (the location of the rooms) to the set $[1,7.5]\times[1,7.5]$, input to the set $[0, 0.5] \times [-0.5, 0.5]$ and create an abstraction\footnote{The parameters used for the abstraction were $\eta=0.25$, $\mu=0.5$, and $\tau=0.5$. An explanation of the meaning of these parameters is given in~\cite{Pessoa}.} using the PESSOA \cite{Pessoa} tool. This abstraction is stored as an Ordered Binary Decision Diagram \cite{BDDbook} (OBDD) and constitutes the game graph describing the dynamics of the cleaning robot. It has 21141 vertices or states and 6 inputs that are available at each state.

We now describe the dynamics of the modes. When the robot is in room $i$ that has not yet been cleaned, the mode can change to the mode where the room $i$ does not need to be cleaned anymore. The nondeterminism in this change models an external signal indicating whether the cleaning in the current room has been completed or not. When all the rooms are cleaned, a nondeterministic mode transition can occur to any other mode to restart the process. In Fig. \ref{modeDynamics}, we illustrate the dynamics of the modes when there are two rooms. As can be seen, there is a nondeterministic transition from $M_3$ to $M_2$ as the robot enters the room 1 ($T_1$). Similarly, if the system is in $M_1$ (only room 1 is not clean), when the robot reaches room 1, the system can take a nondeterministic transition to any of the other modes, i.e., we restart the cleaning process once all the rooms are cleaned. 

\begin{figure}
\begin{center}
\begin{tikzpicture}[shorten >=1pt,node distance=3cm,on grid,auto] 
   \node[state] (M_1)   {$M_1$}; 
   \node[state] (M_2) [right=of M_1] {$M_2$}; 
   \node[state] (M_3) [right=of M_2] {$M_3$}; 
    \path[->] 
    (M_1) edge  [loop above]  node {$\mathtt{true}$} (M_1)
          edge  node [swap] {$T_1$} (M_2)
          edge [bend right] node [swap] {$T_1$} (M_3)
    (M_2) edge  node [swap]  {$T_2$} (M_3)
          edge [loop above] node {$\mathtt{true}$} (M_2)
           edge [bend right]  node [swap] {$T_2$} (M_1)
    (M_3) edge [loop above] node{$\mathtt{true}$} (M_3)
          edge  [bend right] node [swap] {$T_1$} (M_2)
           edge  [bend left=60] node {$T_2$} (M_1);
\end{tikzpicture}
\end{center}
\caption{Mode dynamics for the cleaning robot, when there are two rooms ($M_1$: only room 1 is not clean, $M_2$: only room 2 is not clean, $M_3$ both rooms are not clean).}
\label{modeDynamics}\end{figure}
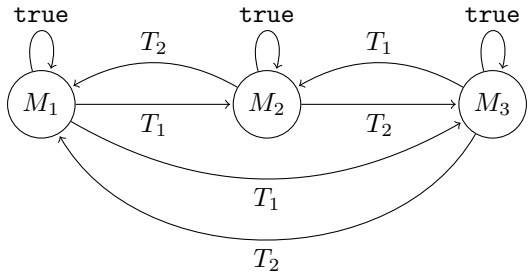

To obtain the final game graph describing the dynamics of both the modes and the cleaning robot, we compose the game graph describing the modes and the game graph describing the dynamics of the robot. Note that, the second player in this game arises due to the conservative nature of the abstraction, as explained in~\cite{Paulobook}, and the nondeterminism in the mode changes, both of which can be modeled as an adversarial disturbance.
\end{subsection}

We compare the performance of the \textbf{GR(1)-Emb}, and the \textbf{MT} algorithms as we increase the number of rooms from 2 to 5. The rooms are boxes of various dimensions defined as:
\begin{eqnarray*}
T_1 &=& \begin{bmatrix}1 &3\end{bmatrix} \times \begin{bmatrix}1 &2.5\end{bmatrix} \\
T_2 &=& \begin{bmatrix}1 &3\end{bmatrix} \times \begin{bmatrix}3 &5\end{bmatrix} \\
T_3 &=& \begin{bmatrix}3.5 &5.5\end{bmatrix} \times \begin{bmatrix}3 &5.5\end{bmatrix} \\
T_4 &=& \begin{bmatrix}3.5 &5.5\end{bmatrix} \times \begin{bmatrix}1 &2.5\end{bmatrix}\\
T_5 &=& \begin{bmatrix}6 &7.5\end{bmatrix} \times \begin{bmatrix}2 &5\end{bmatrix}.
\end{eqnarray*}

Fig. \ref{Fig:cleaningRobot} summarizes our findings. Fig. \ref{fig:subfig2} and Fig. \ref{fig:subfig3} illustrate that, as the number of rooms increases, the gap between the performance of the algorithm \textbf{MT} and the algorithm \textbf{GR(1)-Emb} increases significantly both in terms of number of iterations of the fixed-point algorithms as well as the computation time. Note that, when there are $k$ rooms we have, $m \max_i t_i = (2^k-1)k$, and $\sum_i t_i = \sum\limits_{j=1}^k {k \choose j} k$. Therefore, the widening of the performance gap is expected, since as the number of rooms increases, so does the difference between the worst case time complexities of  \textbf{GR(1)-Emb}  and \textbf{MT}. In terms of memory usage, \textbf{GR(1)-Emb} does slightly worse than \textbf{MT} as expected, but the performance difference is not significant.

\begin{figure}[h]
\centering
 \begin{subfigure}[t]{0.49\textwidth}
\centering
   \includegraphics[scale=0.338,clip=true,trim=35 165 75 210]{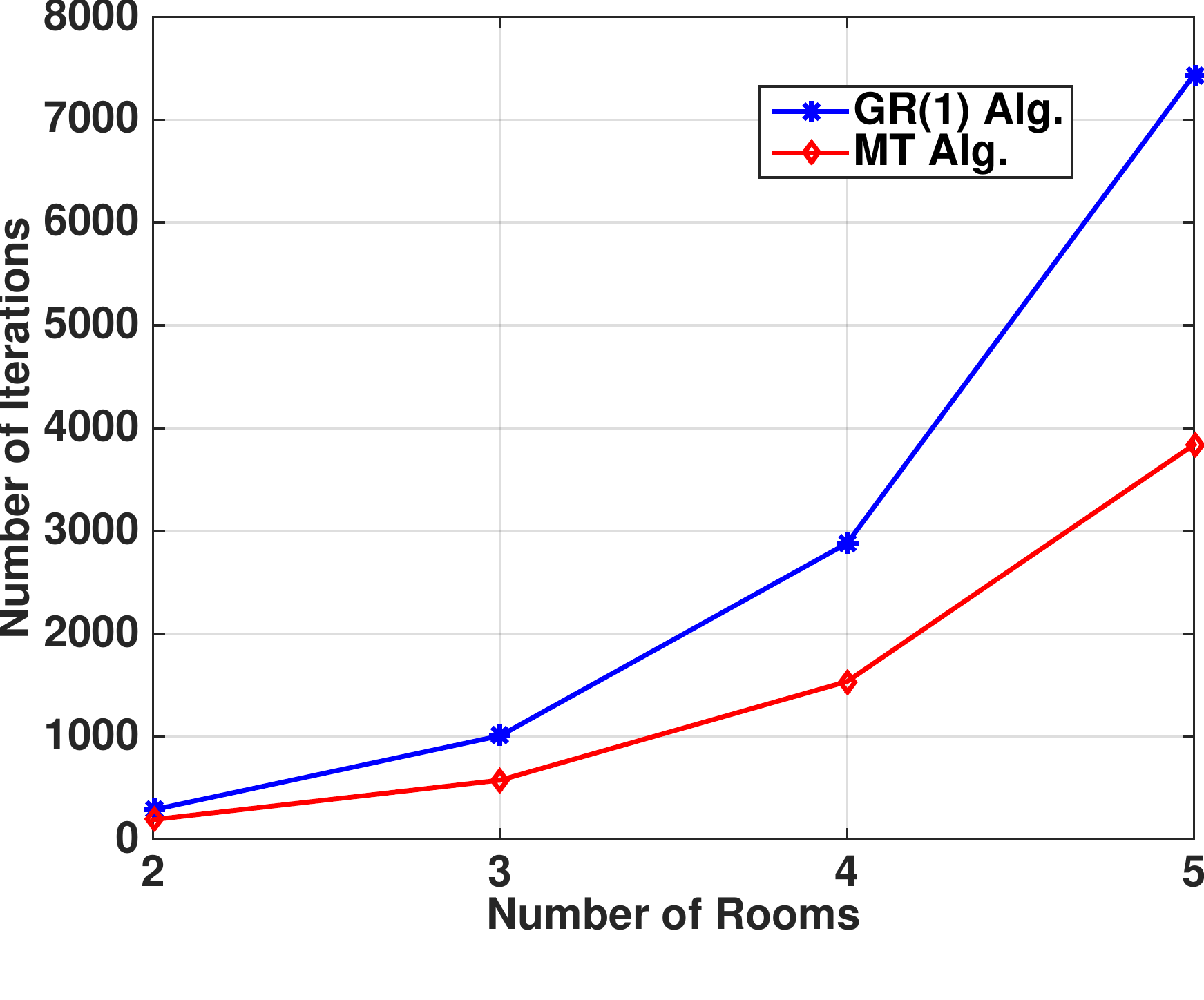}
   \caption{The number of iterations until convergence for the algorithms \textbf{GR(1)-Emb} and \textbf{MT}.}
      \label{fig:subfig2}
 \end{subfigure}
 \hspace{0.25mm}
\begin{subfigure}[t]{0.49\textwidth}
\centering
   \includegraphics[scale=0.338,clip=true,trim=42 165 75 210]{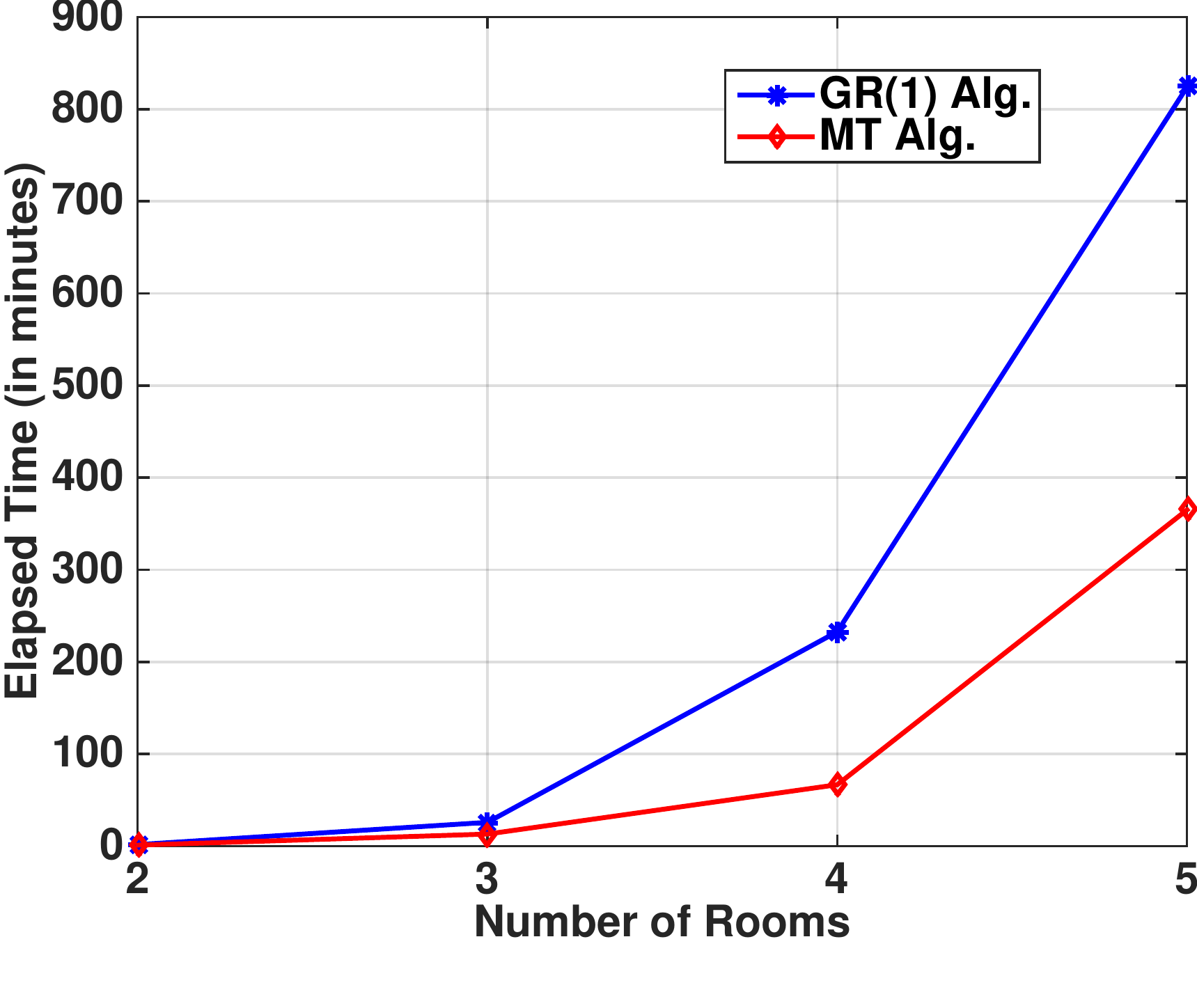}
   \caption{Elapsed time until convergence for the algorithms \textbf{GR(1)-Emb} and \textbf{MT}.}
   \label{fig:subfig3}
 \end{subfigure}
 \hspace{0.2mm}
\caption{Comparison of the algorithms \textbf{GR(1)-Emb} and \textbf{MT} on the cleaning robot case study for varying number of rooms.}
\label{Fig:cleaningRobot}
\end{figure}

\begin{subsection}{Adaptive Cruise Control (ACC)}\label{ACCexample}
The last example demonstrates the usefulness of the MT fragment by applying it on the ACC design problem that we detailed in Section \ref{problemformulation}. We model the dynamics of the ACC equipped vehicle by a hybrid system with two discrete states which specify whether there is a lead car or not. The continuous states describe the evolution of the velocity of the ACC equipped vehicle $(v)$ as well as the velocity of the lead car ($v_L$), and the distance to the lead car ($h$) whenever there is one. The net action of braking and engine torque applied to the wheels ($F_w$) is viewed as the control input and is assumed to satisfy the bound $-0.3mg \leq F_w \leq 0.2mg$, where $m$ is the mass of the ACC equipped vehicle and $g$ is the gravitational constant. Via PESSOA, we constructed a discrete abstraction of this hybrid system, which together with the dynamics of the modes constitutes the game graph of the MT game. The abstraction contains over 1.5 million states. We refer the reader to \cite{ACCpaper2} for the details of the construction of this abstraction and a complete description of the corresponding hybrid model. The winning condition of the game is the conjunction of the safety specification $\varphi_{\text{safety}}$ with the MT formula $\varphi_{\text{speed}} \wedge \varphi_{\text{timegap}}$, where
\begin{equation}
\label{specs_ACC}
\begin{split}
\varphi_{\text{safety}} & \equiv \square [\tau \geq \tau_{\text{safe}}], \\
\varphi_{\text{speed}} &\equiv \left(\Diamond \square M_{\text{speed}} \implies \Diamond \square [v_{\text{des}}-\epsilon_v, v_{\text{des}}+\epsilon_v]\right),\\
\varphi_{\text{timegap}}& \equiv \left(\Diamond \square M_{\text{timegap}} \implies \Diamond \square [\tau_{\text{des}}- \epsilon_\tau, \tau_{\text{des}}+\epsilon_\tau]\right).
\end{split}
\end{equation}
The values of the parameters appearing in \eqref{specs_ACC} are $\tau_{\text{safe}} = 1$ s, $v_{\text{des}} = 25$ m/s, $\epsilon_v$ = 1,  $\tau_{\text{des}}= 1.6$, and $\epsilon_\tau =1$. Note that the additional safety formula, $\varphi_{safety}$, can be handled separately by first synthesizing a safety controller and then composing this controller with a controller synthesized solely for the MT formula, $\varphi_{\text{speed}} \wedge \varphi_{\text{timegap}}$.

In Figure \ref{ACCresults}, we present the winning set computed via the \textbf{MT} Algorithm. As can be seen, the domain does not contain the points where $h$, the headway, is small and $v$, the velocity of the ACC vehicle, is high, since there is no sequence of control inputs to maintain a safe headway starting from these states. We simulated the MT controller on CARSIM, an industry standard car dynamics simulation package, for the following scenario: at time $t = 0$ s, a lead car is present driving below the desired speed $v_{\text{des}} = 25$ m/s of the ACC car, then leaves the lane at $t = 3$ s, allowing the ACC car to reach and attain its desired speed. At $t = 13$ s, a new lead car cuts in $30$ m in front of the ACC car and starts decelerating. This means that the ACC car should slow down in order to increase the headway. Fig. \ref{ACCresults2} presents the behavior of the MT controller. Notably, all constraints, which are indicated by green lines, are satisfied throughout the simulations. For a detailed discussion on hardware implementation of the MT controller and further experimental results, we refer the reader to \cite{ACCpaper2}.
\begin{figure}
\centering
   \includegraphics[scale=0.9,clip=true,trim=0 1 0 1]{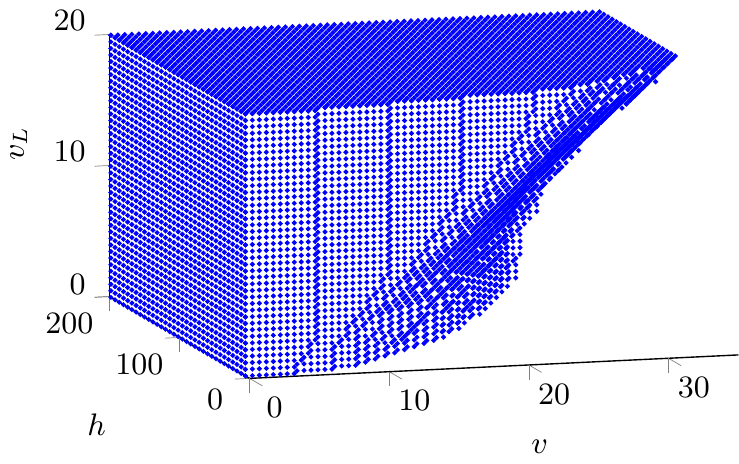}
      \caption{The winning set computed by the \textbf{MT} Algorithm.}
   \label{ACCresults}
\end{figure}

\begin{figure}[h]
\centering
 \includegraphics[scale=0.7,clip=true,trim=0 1 0 1]{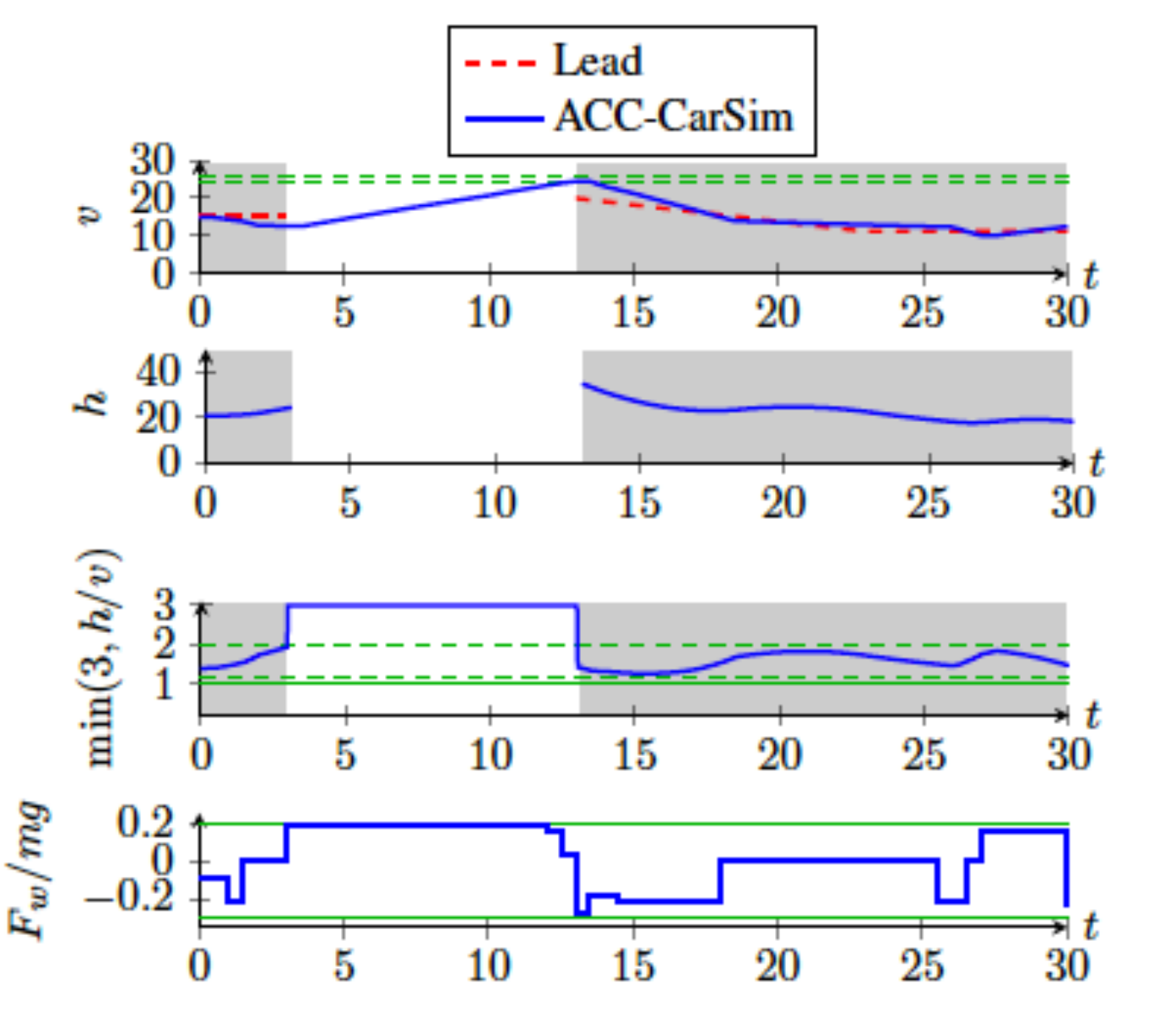}
 	\caption{Simulation results in CarSim of the PESSOA controllers. The plots show, from top to bottom, velocities, headway, time headway, and applied control input. Grayed areas indicate that the system is in specification mode $M_{\text{timegap}}$. Dashed green lines indicate target sets, solid green indicate safety sets.}
	\label{ACCresults2}
\end{figure}

\end{subsection}
The different experimental results suggests the following: \textbf{(1)} \textbf{MT} is consistently better than \textbf{GR(1)-Emb}. Even for the case when the theoretical worst case complexities of both algorithms are the same, \textbf{MT} outperforms \textbf{GR(1)-Emb}. However, the performance increase is not always considerable in this case; \textbf{(2)} there is no significant difference in the memory usage between \textbf{MT} and \textbf{GR(1)-Emb} algorithms; \textbf{(3)} as the gap between $\max_i t_i$ and $\sum_i t_i$ widens, so does the performance difference between \textbf{GR(1)-Emb} and \textbf{MT}, which is in accordance with the results in Section \ref{compWinningSet}.

%
%
%
%
%
%
%
%


\section{Conclusions}\label{conclusion}
We introduced a new class of LTL games called mode-target games and argued that these games can be used to model a variety of control design problems encountered in practice. We provided two algorithms to solve MT games. The first algorithm is based on transforming MT games to simple games, a class of LTL games for which we provide a synthesis algorithm. This leads to an algorithm that solves MT games in a number of steps polynomial in the size of the game graph. We next provided a different algorithm, that relies on the fact that every MT game can be embedded into a GR(1) game. We also showed that the direct algorithm has better worst case complexity than the algorithm obtained via the GR(1) embedding. These observations were validated through multiple simulations. As future work, we plan on investigating whether additional structure arising in control problems can lead to further simplifications both in MT games as well as other LTL games.

\section{Acknowledgements}
\label{Acknowledgements}
The work is supported by the NSF Contract \#CNS-1239037 and the project ExCAPE: Expeditions in Computer Augmented Program Engineering. The authors would like to thank Omar Hussien for his help with the CarSim simulations.

\appendix
\section{Preliminary Lemmas}\label{pre_lemmas}
A property $\Phi$ is a \emph{stable property} iff $\Post(\Phi) \subseteq \Phi$,
i.e., if $\Phi$ is closed under suffixes. We call $\varphi$ a \emph{stable}
formula if $W(\varphi)$ is a stable property.
It is proved in \cite{Sistla} that a formula $\varphi$ is a stable formula
iff $\square \varphi \equiv \varphi$. Then it
follows that any formula of the form $\square \phi$, for some $\phi$ is a stable formula.
Moreover, the conjunction of stable formulas is also a stable formula.
Take two stable formulas $\varphi_1$ and $\varphi_2$; then $\varphi_1 \wedge \varphi_2 \equiv \square \varphi_1 \wedge \square \varphi_2 \equiv \square (\varphi_1 \wedge \varphi_2)$, which is a stable formula.
Also recall that a property $\Phi$ is an \emph{absolute liveness property} iff
$\Sigma^*\Phi \subseteq \Phi$. We call $\varphi$ an \emph{absolute liveness formula}
if $W(\varphi)$ is an absolute liveness property.

%
%
%
\begin{lemma}\label{implicationPre}Given the formulae $\varphi_1$ and $\varphi_2$, if we have $W_G( \varphi_1 \wedge \varphi_2) = \emptyset$, then the following holds:
\begin{equation*}W_G(\neg \varphi_1 \vee \varphi_2) = W_G(\neg \varphi_1).
\end{equation*}
\end{lemma}

\begin{proof}
\begin{eqnarray*}W_G(\neg \varphi_1 \vee \varphi_2)
&=& W_G((\neg \varphi_1 \wedge \varphi_2) \vee (\neg \varphi_1 \wedge \neg \varphi_2) \vee (\varphi_1 \wedge \varphi_2)) \\
&=& W_G((\neg \varphi_1 \wedge \varphi_2) \vee (\neg \varphi_1 \wedge \neg \varphi_2)) \\
&=& W_G(\neg \varphi_1).
\end{eqnarray*}
\end{proof}

\begin{lemma}\label{implicationProof}Given the sets of LTL formulae $\cup_{i \in I} \{ \varphi_i \}$, $\cup_{i \in I} \{ \psi_i \}$, and a game graph $G$, if for all $i \in I$ we have $W_G\left(\varphi_i \wedge \bigvee_{j \in I \setminus \{i\}}\psi_j\right) = \emptyset$, then the following holds:
\begin{equation*}W_G\left(\bigvee\limits_{i \in I}\varphi_i \implies \bigvee\limits_{i \in I} \psi_i \right) = W_G\left(\bigwedge\limits_{i \in I} \left(\varphi_i \implies \psi_i \right)\right)
%
\end{equation*}
\end{lemma}

\begin{proof}The following holds:
\begin{eqnarray*}W_G\left(\bigvee\limits_{i \in I}\varphi_i \implies \bigvee\limits_{i \in I} \psi_i \right) &=& W_G\left(\left(\bigwedge\limits_{i \in I}\neg \varphi_i\right) \vee \left(\bigvee\limits_{j \in I} \psi_j \right)\right) \\
&=& W_G\left(\left(\bigwedge\limits_{i \in I}\neg \varphi_i \vee \psi_i \right) \vee \left(\bigwedge\limits_{i \in I} \neg \varphi_i \vee \bigvee\limits_{j \in I \setminus\{i\}} \psi_j \right)\right) \\
&=&  W_G\left(\bigwedge_{i \in I} (\neg\varphi_i \vee \psi_i)\right) \cup W_G\left(\bigwedge\limits_{i \in I} \left(\neg \varphi_i \vee \bigvee\limits_{j \in I \setminus\{i\}} \psi_j \right)\right) \\
& \eqbir & W_G\left(\bigwedge_{i \in I} (\neg\varphi_i \vee \psi_i)\right) \cup W_G\left(\bigwedge\limits_{i \in I} \neg \varphi_ i \right) \\
& \eqiki & W_G\left(\bigwedge_{i \in I} (\neg\varphi_i \vee \psi_i)\right) = W_G\left(\bigwedge\limits_{i \in I} \left(\varphi_i \implies \psi_i \right)\right) 
\end{eqnarray*}
where $\eqbir$ follows from the fact that $\forall i \in I, W_G\left(\varphi_i \wedge \bigvee_{j \in I \setminus \{i\}}\psi_j\right) = \emptyset$, and Lemma~\ref{implicationPre}, while $\eqiki$ follows from the inclusion $W_G\left(\bigwedge_{i \in I} \neg \varphi_i\right) \subseteq W_G\left(\bigwedge_{i \in I} (\neg\varphi_i \vee \psi_i)\right).$
%
%
%
%

\end{proof}

\begin{lemma}\label{always2} Given a stable formula $\varphi$, and a winning strategy $f$ for player 0 in $(G,\varphi)$, we have $\llbracket \varphi \rrbracket = V^*$, where $V^* := \cup_{v \in \llbracket \varphi \rrbracket} \cup_{r \in \Omega_{f,v}(G)}\cup_{i \in \N}r_i$, i.e., the set of all states visited under the strategy $f$.
\end{lemma}

\begin{proof}We note that it suffices to show $V^* \subseteq \llbracket \varphi \rrbracket$. The other direction is immediate due to the definition of $V^*$. Note that since $\varphi$ is a stable formula, it is closed under suffixes. This means that any strategy $f$ that is winning for $(G, \varphi)$ is winning for
$(G, \square \llbracket \varphi \rrbracket)$ as well. Therefore, any play $r \in \cup_{v \in \llbracket \varphi \rrbracket} \cup_{r \in \Omega_{f,v}(G)}$ always stays inside the set $\llbracket \varphi \rrbracket$, hence  $V^* \subseteq \llbracket \varphi \rrbracket$, and the result follows.
\end{proof}

\begin{lemma}\label{syntacticequiv}Let $p$ and $q$ be positional formulas, then
\[\square (\Diamond p \vee \Diamond \square q)\equiv \square \Diamond p \vee \Diamond \square q.\]
\end{lemma}

\begin{proof}
\begin{align*}\square (\Diamond p \vee \Diamond \square q)
 &\equivbir \square (\Diamond (p \vee \square q))
\equiv \square \Diamond (p \vee \square q) \\
&\equiviki \square \Diamond p \vee \square \Diamond \square q \equivuc \square \Diamond p \vee \Diamond \square q,
\end{align*}
where $\equivbir$ holds since $\Diamond \varphi_1 \vee \Diamond \varphi_2 \equiv \Diamond (\varphi_1 \vee \varphi_2)$,and $\equiviki$ is true because $\square \Diamond (\varphi_1 \vee \varphi_2) \equiv \square \Diamond \varphi_1 \vee \square \Diamond \varphi_2$. Finally, $\equivuc$ follows from $\square \Diamond \square q \equiv \Diamond \square q$.\\
\end{proof}

\begin{lemma}\label{syntacticequiv2}Let $p$ and $q$ be
positional formulas, then \[(\Diamond \square p \implies
\Diamond \square q) \equiv (\Diamond \square p \implies \Diamond \square (p \wedge q)).\]
\end{lemma}
\begin{proof}
\begin{align*}&(\Diamond \square p \implies \Diamond \square (p \wedge q))  
\equiv (\Diamond \square p \implies (\Diamond \square p \wedge \Diamond \square q))
\\&\equiv (\neg (\Diamond \square p) \vee (\Diamond \square p \wedge \Diamond \square q))\\
&\equivbir (\neg (\Diamond \square p) \vee \Diamond \square p) \wedge (\neg (\Diamond \square p) \vee \Diamond \square q) \\
& \equiv \textbf{True} \wedge (\neg (\Diamond \square p) \vee \Diamond \square q) \equiv (\Diamond \square p \implies \Diamond \square q),
\end{align*}
where $\equivbir$ holds because $\vee$ distributes over $\wedge$.
\end{proof}

\section{Proof of Lemma 1}
\label{appendix3}
Given a GR(1) formula $\varphi$, the following holds:
\begin{equation*}
\varphi \equiv \bigvee\limits_{i_1 \in I_1} \Diamond \square \neg a_i \vee \bigwedge\limits_{i_2 \in I_2} \square \Diamond g_{i_2} \equivbir \square \bigwedge\limits_{i_2 \in I_2}\left(\left(\bigvee\limits_{i_1 \in I_1} \Diamond \square \neg a_{i_1}\right) \vee \Diamond g_{i_2}\right),
\end{equation*}
where $\equivbir$ follows from very similar arguments to those in the proof of Lemma \ref{syntacticequiv} in Appendix \ref{pre_lemmas}.
Note that $\vee_{i_1 \in I_1} \Diamond \square \neg a_{i_1}$ implies $\varphi$, i.e., $W\left(\vee_{i_1 \in I_1} \Diamond \square \neg a_{i_1}\right) \subseteq W(\varphi).$ Therefore, \mbox{$\varphi \equiv \square \left( \wedge_{i_2 \in I_2} \Diamond g_{i_2} \vee \psi_{i_2}\right)$,} where $\psi_{i_2}:=\vee_{i_1 \in I_1} \Diamond \square \neg a_{i_1},$ for all $i_2 \in I_2$, which completes the proof of the lemma.

\section{Proof of Lemma \ref{lemma:mdfragment}}
\label{app2proof1}
\begin{eqnarray*}
\varphi=\bigwedge\limits_{i=1}^m \left(\Diamond \square M_i \implies
\bigvee\limits_{j=1}^{t_i}\Diamond \square T_{i,j}\right)
&\equivbir& \bigwedge\limits_{i=1}^m \left(\square \Diamond \neg M_i \vee \bigvee\limits_{j=1}^{t_i}\Diamond \square (M_i \wedge T_{i,j})\right)\\
&\equiviki& \square \bigwedge\limits_{i=1}^m \left(\Diamond \neg M_i \vee \bigvee\limits_{j=1}^{t_i}\Diamond \square (M_i \wedge T_{i,j})\right),
\end{eqnarray*}
where $\equivbir$ is due to Lemma \ref{syntacticequiv2}, while $\equiviki$ follows from Lemma \ref{syntacticequiv} and \linebreak $\square \varphi_1 \wedge \square \varphi_2 \equiv \square (\varphi_1 \wedge \varphi_2)$.

The last formula has the form given in the statement of the lemma, where $p_i$ is $\neg M_i$ and
$\psi_{i}$ is $\vee_{j=1}^{t_i}\Diamond \square \left(M_i \wedge T_{i,j}\right)$. Then, we are only left with showing that $W_G(\psi_i) \subseteq W_G(\varphi)$.

Recall that in MT games for all $v \in V$, if $M_i \in L(v)$ then \mbox{$M_j \not\in L(v)$}
for all $j \not=i$. It follows that for any $r \in V^\omega$ we have:
$L(r) \models \Diamond \square M_i \implies L(r) \models \Diamond \square \neg M_j$, for all $j \not = i$.
Moreover, note that $W(\Diamond \square \neg M_j) \subseteq W(\square \Diamond \neg M_j)$. Therefore, the following holds:
\begin{align}\label{inequal2}\begin{split}W_G\left(\bigvee\limits_{j=1}^{t_i}\Diamond \square (M_i \wedge T_{i,j})\right)
&\subseteq W\left(\square\bigwedge\limits_{\substack{\ell \in I_{\setminus i}}}
\Diamond \neg M_\ell\right)\, \text{where $I_{\setminus i}=\{1,2, \ldots, m\} \setminus \{i\}$}\\
& \subseteq  W\left(\square \left(\bigwedge\limits_{\ell \in I_{\setminus i}}
\Diamond \neg M_\ell \vee \bigvee\limits_{j=1}^{t_\ell}\Diamond \square (M_\ell \wedge T_{\ell,j})\right)\right).
\end{split}
\end{align}
Also note that
\begin{equation}\begin{split}\label{inequal1}W_G\left(\bigvee\limits_{j=1}^{t_i}\Diamond \square
(M_i \wedge T_{i,j})\right) &\subseteq W_G\left(\square \Diamond \neg M_i \vee
\bigvee\limits_{j=1}^{t_i}\Diamond \square (M_i \wedge T_{i,j})\right)\\
&=W_G\left( \square \left(\Diamond \neg M_i \vee \bigvee\limits_{j=1}^{t_i}\Diamond\square (M_i \wedge T_{i,j})\right)\right),
\end{split}\end{equation}
where the last equality is due to Lemma \ref{syntacticequiv}.

By combining the inclusions (\ref{inequal1}) and (\ref{inequal2}) we get
 \begin{eqnarray*} W_G\left(\bigvee\limits_{j=1}^{t_i}\Diamond \square (M_i \wedge T_{i,j})\right) \subseteq W_G\left(\square\bigwedge\limits_{i=1}^m\left( \Diamond \neg M_i \vee \bigvee\limits_{j=1}^{t_i}\Diamond \square (M_i \wedge T_{i,j})\right)\right),\end{eqnarray*}
which completes the proof of the lemma.

\section{Proof of Theorem \ref{theorem:compsound}}
\label{soundness_app}
Let $Z^*=\nu Z \bigcap\limits_{i \in I} \llbracket\psi_i \vee \Diamond (p_i \wedge \Circle Z) \rrbracket$.
We start by proving $Z^* \subseteq \left\llbracket \square \wedge_{i \in I} \varphi_i \right\rrbracket$.
We make the following observation:
\begin{small}
\begin{equation*}
\left((\Sigma^* p_1)
(\Sigma^* p_2)\ldots(\Sigma^* p_{|I|})\right)^\omega
=W\left(\square \bigwedge\limits_{i \in I} \Diamond p_i\right) \subseteq W\left(\square \bigwedge\limits_{i \in I} \Diamond p_i \vee \psi_i\right).
\end{equation*}
\end{small}
This suggests that a strategy that visits all $p_i$'s in a circular
fashion is winning for player $0$. We pick the visiting order
$p_1p_2\ldots p_i\ldots p_{|I|}$, since it is enough to find
one winning strategy. Therefore, whenever a play visits a state
that satisfies $p_i$ player $0$ should be able to switch to a
strategy that is winning for the game with the winning condition $\Diamond p_{i+1 (mod |I|)}$.
Next, we explain that this is in fact possible on $Z^*$.

The game starts at a state in $Z^*$. Player $0$ follows the strategy that is 
winning for the game $(G, \psi_i \vee \Diamond (p_i \wedge \Circle Z^*))$, from $Z^*$.
If the game reaches a state $v \in \llbracket p_i \rrbracket$, then player $0$ forces a visit to $Z^*$.
After that player $0$ starts following a strategy that is winning for the game with the winning condition:
$\psi_{i+1 (mod |I|)} \vee \Diamond(p_{i+1 (mod |I|)} \wedge \Circle Z^*).$
This switching is possible since $Z^* \subseteq \llbracket \psi_i \vee \Diamond (p_i \wedge \Circle Z^*)\rrbracket$, for all $i \in I$.
The circular switching can be implemented using a counter, with $|I|$ states.

Due to the disjunction of the reachability part of the formula with $\psi_i$, it is true 
that a play that follows the above strategy can be winning for 
$(G, \psi_i)$ for some $i \in I$, instead of $(G, \Diamond p_i)$ for some $i \in I$.
However, since we assumed that for each $i \in I$, $\psi_i$ is an absolute liveness formula, and 
\mbox{$W_G(\psi_i) \subseteq
W\left(\square\wedge_{i \in I} \varphi_i\right)$,}
even in this case the play is winning for $\square \bigwedge\limits_{i \in I} \varphi_i$. Therefore, $Z^* \subseteq \left\llbracket \square \wedge_{i \in I} \varphi_i \right\rrbracket$.

Now, we show that the other direction, i.e., \mbox{$\left\llbracket \square \wedge_{i \in I} \varphi_i \right\rrbracket \subseteq Z^*$.}
\vspace{1mm}
To show that $\left\llbracket \square \wedge_{i \in I} \varphi_i \right\rrbracket \subseteq Z^*$, it is sufficient to show $\left\llbracket \square \wedge_{i \in I} \varphi_i \right\rrbracket \subseteq F\left(\left\llbracket \square \wedge_{i \in I} \varphi_i \right\rrbracket\right)$, where $F(Z):=\cap_{i \in I} \left\llbracket\psi_i \vee \Diamond (p_i \wedge \Circle Z) \right\rrbracket$ (see e.g. \cite{tarski}).
Since $\square \wedge_{i \in I} \varphi_i$ is a stable formula, we can invoke Lemma~\ref{always2}, with $\varphi=\square \wedge_{i \in I} \varphi_i$ and conclude that \begin{equation*}\llbracket \square \wedge_{i \in I} \varphi_i \rrbracket = V^*,\end{equation*}
where $V^* = \cup_{v \in \llbracket \square \wedge_{i \in I} \varphi_i \rrbracket} \cup_{r \in \Omega_{v,f}(G)} \cup_{i \in \N} r_i$.
\begin{eqnarray*}
\llbracket \square \wedge_{i \in I} \varphi_i \rrbracket = V^* &\incbir& \bigcap\limits_{i \in I} \left\llbracket (\psi_i \vee \Diamond p_i) \wedge \square V^*\right\rrbracket \\
&\subseteq&  \bigcap\limits_{i \in I} \left\llbracket (\psi_i \vee \Diamond (p_i \wedge \Circle V^*) \right\rrbracket,
\end{eqnarray*}
where $\incbir$ follows from the definition of $V^*$, since it includes all states visited under the winning strategy for player 0 in $(G, \varphi)$. We just proved that $V^* \subseteq F(V^*)$. Note that, for any $S \subseteq V$ we have $S \subseteq F(S) \implies S\subseteq Z^*$ due to \cite{tarski}. This shows that $V^* = \llbracket \square \wedge_{i \in I} \varphi_i \rrbracket \subseteq Z^*$, which completes the proof.
%

\section{Strategy Synthesis}
\label{synthesis}
Recall that a strategy is a partial function \mbox{$f:V^*\times V_0\to V$} such that whenever $f(r,v)$ is defined, $(v,f(r,v))\in E$. We next construct a memoryless strategy \mbox{$f:V_0\to V$} based on a set of edges that can be computed from the intermediate results obtained when computing the fixed-point in (\ref{finalfixed-point}).
%

We start with some additional notation.
We use $Y_i^{*\ell}$ to denote the set computed at the $\ell^{th}$ iteration of the following fixed-point computation over $Y$:
\begin{equation*}\hspace{-0.8mm}\mu Y \hspace{-0.3mm} \left(\bigcup\limits_{j=1}^{t_\ell}\hspace{-0.15mm}\nu X \hspace{-0.2mm} (\Pre(X) \cap \llbracket M_i \wedge T_{i,j}\rrbracket) \right.
\left.\hspace{-0.2mm} \cup (\llbracket \neg M_i \rrbracket \cap \llbracket \varphi \rrbracket) \hspace{-0.2mm}\cup\hspace{-0.3mm} \Pre(Y) \hspace{-4mm}\phantom{\bigcup\limits_j}\right)\hspace{-0.2mm}.\end{equation*}
Similarly, $X_{i,j}^{*\ell*}$ denotes
\begin{equation*}\nu X (\Pre(X) \cap \llbracket M_i \wedge T_{i,j} \rrbracket) \cup (\llbracket \neg M_i \rrbracket \cap \llbracket \varphi \rrbracket) \cup \Pre(Y_i^{*\ell})).\end{equation*} To simplify the construction of the strategy, without loss of generality we assume that the modes are exhaustive, i.e., $\cup_i \llbracket M_i \rrbracket = V$. For each mode $M_k$, where \mbox{$k \in \{1,2, \ldots, m\}$,} we define the set of edges \mbox{$E_k:=E_{k,1} \cup E_{k,2}$}  such that:
\begin{align*}
E_{k,1} &= \bigcup\limits_{\ell>1}\left\{(v, v') \in E\right.  \vert \left. v \in Y_k^{*\ell} \land v \not\in Y_k^{*<\ell} \land v' \in Y_k^{* < \ell} \right\} ,\notag\\
E_{k,2} &= \bigcup\limits_{j=1}^{t_k}\bigcup\limits_{\ell}\left\{ (v,v') \in E\right. \vert  v \in X_{k,j}^{*\ell*} \cap \llbracket M_k \wedge T_{k,j}\rrbracket \left.\land\, v \not\in X_{k,j}^{*<\ell*} \land v' \in X_{k,j}^{*\ell*}\right\},\notag
\end{align*}
where $Y_k^{*< \ell}=\bigcup\limits_{0 \leq i < \ell} Y_k^{*i}$ and $X_{k,j}^{*<\ell*}=\bigcup\limits_{0 \leq \ell < k} X_{k,j}^{*\ell*}$. 
$E_{k,1}$ corresponds to the transitions, that player $0$ can force the game to make progress towards a state in $\llbracket \neg M_k \rrbracket$ or a state that will not leave $\llbracket M_k \wedge T_{k,j}\rrbracket$ forever for some $j$. The edges in $E_{j,2}$ are the transitions, where the game is at a state in $\llbracket M_k \wedge T_{k,j} \rrbracket$, and player $0$ can force the game to stay in $\llbracket M_k \wedge T_{k,j} \rrbracket$ but cannot force it to make progress towards a state in $\llbracket \neg M_k \rrbracket$. Note that player 0 still wins by always taking the transitions in $E_{j,2}$ since even if there is no progress towards $\llbracket \neg M_k \rrbracket$, the game stays in $\llbracket M_k \wedge T_{k,j}\rrbracket$ forever as well.
%
As a final step, we use of edges $E_k$ to define $f: \llbracket \varphi \rrbracket \to V$ as
$f(v_0)=v',$ where $v_0 \in \llbracket M_k \rrbracket, v' \in \llbracket \varphi \rrbracket$ and $(v_0, v') \in E_k$, which completes the construction of the winning strategy.
%
%
%

\section{Proof of Proposition \ref{prop_embed1}}
\label{GR1incMT}
\label{embedding1}
We prove this proposition in two main steps. In the first step, we show that every mode-target game can be transformed into an equivalent GR(1) game.

Let $(G, \varphi)$ be a mode-target game. Then the following holds:
\begin{equation}
 \label{mtequiv}
\begin{split}
\varphi=\bigwedge\limits_{i=1}^m \left(\Diamond \square M_i \implies
\bigvee\limits_{j=1}^{t_i}\Diamond \square T_{i,j}\right) &\equivbir \bigwedge\limits_{i=1}^m \left(\Diamond \square M_i \implies
\bigvee\limits_{j=1}^{t_i} \Diamond \square (M_i \wedge T_{i,j})\right) \\
\end{split}
\end{equation}

Let $\varphi_i = \Diamond\square(M_i \wedge T_i)$ and $\psi_i = \bigvee_{j=1}^{t_i}\Diamond \square (M_i \wedge T_{i,j})$. Since the modes are mutually exclusive, i.e., $M_i \in L(v) \implies M_j \not\in L(v),\, \forall\, j \not=i$ and $\forall v \in V,$ the following holds:
\begin{equation}\label{disjoint_res1}W_G\left(\varphi_i \wedge \bigvee_{i \in I \setminus \{j\}} \psi_j\right) = \emptyset, \forall i \in I.\end{equation}
Then due to Lemma \ref{implicationProof} we get:
 $\varphi \equiv \left(\bigvee_{i=1}^m \Diamond \square M_i \implies \bigvee_{i =1}^m \bigvee_{i=1}^{t_i} \Diamond \square (M_i \wedge T_{i,j})\right).$ Next we show:
\begin{eqnarray*}
&W_G\left(\bigvee_{i=1}^m \Diamond \square M_i \implies \bigvee_{i =1}^m \bigvee_{j=1}^{t_i} \Diamond \square (M_i \wedge T_{i,j})\right) \\
&\hspace{20mm}= W_G \left(\bigvee_{i=1}^m \Diamond \square M_i \implies \bigvee_{j =1}^{\max_i t_i} \Diamond \square \vee_{i=1}^{m} (M_i \wedge \bar{T}_{i,j})\right) 
\end{eqnarray*}
where $\bar{T}_{i,j} = T_{i,j}$ if $j \leq t_i$ and $\bar{T}_{i,j} = \mathtt{false}$, otherwise.

The inclusion:
\begin{equation*}W_G(\varphi) \subseteq W_G \left(\bigvee_{i=1}^m \Diamond \square M_i \implies \bigvee_{j =1}^{\max_i t_i} \Diamond \square \vee_{i=1}^{m} (M_i \wedge \bar{T}_{i,j})\right) 
\end{equation*}
is immediate since $\bigvee\limits_{i=1}^m \bigvee\limits_{j=1}^{t_i} \Diamond\square (M_i \wedge T_{i,j}) \equiv \bigvee\limits_{j=1}^{\max_i t_i} \bigvee\limits_{i=1}^m \Diamond \square (M_i \wedge \bar{T}_{i,j})$ and \linebreak \mbox{$ \bigvee\limits_{j=1}^{\max_i t_i} \bigvee\limits_{i=1}^m \Diamond \square (M_i \wedge \bar{T}_{i,j})$} implies $ \bigvee\limits_{j=1}^{\max_i t_i} \Diamond \square \vee_{i=1}^m (M_i \wedge \bar{T}_{i,j})$. To show the other direction, we start with the following observation. Suppose $r \in V^\omega$, and let $I$ be a finite index set. Then the following semantic relation holds:
\begin{equation}\label{dis_eventuallyalways}L(r) \models \square \bigvee\limits_{i \in I} p_i\, \implies L(r) \models \bigvee\limits_{i \in I} \square p_i \vee \bigvee\limits_{\substack{J \subseteq I,\\ |J|>1}}\bigwedge\limits_{j \in J} \square\Diamond p_j, 
\end{equation}
where each $p_i$ is a positional formula. Note that this follows from the fact that any word satisfying $\square \vee_{i \in I} p_i$ should either always stay in one of the $p_i$'s forever, and hence satisfy $\vee_{i \in I}\square p_i$ or shuffle between at least two different $p_i$'s, i.e., satisfy $\bigvee_{\substack{J \subseteq I,\\ |J|>1}}\bigwedge_{j \in J} \square\Diamond p_j$.
Let $I:=\{1, 2, \ldots m\}$. We are now ready to show the other direction as follows:
\begin{eqnarray*}
&\hspace{-60mm}W_G \left(\bigvee\limits_{i=1}^m \Diamond\square M_i \implies \bigvee\limits_{j=1}^{\max_i t_i}\Diamond \square \bigvee\limits_{i=1}^m (M_i \wedge \bar{T}_{i,j})\right) \\
&\hspace{-4mm}\incbir W_G\left(\bigvee\limits_{j=1}^{\max_i t_i} \bigvee\limits_{i=1}^{m} \Diamond \square (M_i \wedge \bar{T}_{i,j})\right. \vee \bigvee\limits_{j=1}^{\max_i t_i} \bigvee\limits_{i=1}^{m} \bigvee\limits_{\substack{J \subseteq I_m,\\ |J|>1}}\bigwedge\limits_{s \in J}  \square\Diamond (M_{s} \wedge \bar{T}_{s,j}) 
\left. \vee \bigwedge\limits_{i=1}^m \square \Diamond \neg M_i\right) \\
&\hspace{-60mm} \inciki W_G\left(\bigvee\limits_{i=1}^m \bigvee\limits_{j=1}^{\max_i t_i} \Diamond \square (M_i \wedge \bar{T}_{i,j}) \vee \bigwedge\limits_{i=1}^m \square \Diamond \neg M_i \right),
\end{eqnarray*}
where $\incbir$ follows from the inclusion given in (\ref{dis_eventuallyalways}), distributivity of $\Diamond$ with respect to $\vee$ and the syntactic equivalence
\begin{equation*}\Diamond \bigwedge\limits_{s \in J} \square \Diamond (M_{s} \wedge \bar{T}_{s,j})  \equiv \bigwedge\limits_{s \in J} \square \Diamond (M_{s,j} \wedge \bar{T}_{s,j}).\end{equation*}
Due to the disjointness of modes we have $W_G(M_i) \subseteq W_G(\neg M_j), \,\forall j \not = i$, and therefore $\inciki$ follows from the fact that $\bigvee\limits_{\substack{J \subseteq I_m,\\ |J|>1}}\bigwedge\limits_{s \in J}  \square\Diamond (M_{s} \wedge \bar{T}_{s,j})$ implies $\wedge_{i=1}^m \square \Diamond \neg M_i$.
Therefore we have:
\begin{equation}\label{eq:proofLast}W_G(\varphi) = W_G \left(\bigvee\limits_{i=1}^m \Diamond\square M_i \implies \bigvee\limits_{j=1}^{\max_i t_i} \Diamond \square \vee_{i=1}^m (M_i \wedge \bar{T}_{i,j})\right).
\end{equation}
This completes the proof since we can rewrite the formula on the right hand side of the equality \eqref{eq:proofLast} and get the equality in \eqref{Eq:Emb1}.

\section{Description of The Mode and Target Sets in Section \ref{exampleLTI}}
\label{exampleData}
\begin{lstlisting}
for num_of_extra_targets=0:2:10
    while (loop_counter_aux<=3)     
        % Continuous dynamics (linear)
        a0=-1;
        b0=1;
        %create matrices where each element is between -1 and 1
        system1_cont.A= (b0-a0).*rand(2,2) + a0;
        system1_cont.B=(b0-a0).*rand(2,1) + a0;

        for num_of_modes=[3:3:9]
            num_of_targets=num_of_modes+num_of_extra_targets-1;
            
            mode_set_man{1}=[-5 -3.25;-1 2];
            mode_set_man{2}=[-2 0.2;1 4];
            mode_set_man{3}=[3 6;-2 -0.25];
            mode_set_man{4}=[-5 -2.5;3.25 5];
            mode_set_man{5}=[3.5 5;0 2.5];
            mode_set_man{6}=[0 2;4.5 6];
            mode_set_man{7}=[-2 0;-2 -1.25];
            mode_set_man{8}=[-2 0;-1 0.5];
            mode_set_man{9}=[0.25 2;-1 0.5];
            mode_set_man{10}=[0.7 2;1.5 4];
            mode_set_man{11}=[3.5 5;-6 -3];
            mode_set_man{12}=[-6 -3;-6 -3];
            mode_set_man{13}=[0.25 2;-6 -3];
            mode_set_man{14}=[3 6;3 6];
            mode_set_man{15}=[-2 0;-6 -3];
            
            P = randperm(15,num_of_modes);
            a0=0.75;
            b0=0.8;
            target_r = (b0-a0).*rand(1,1) + a0;
            
            %center of the mode sets
            for i=1:num_of_modes
                mode_set_center{i}=mode_set{i}(:,1)+((mode_set{i}(:,2)...
                    -mode_set{i}(:,1))/2);
                %first initialize numb of targets per each mode to 1
                num_of_targets_per_mode{i}=1;
            end
            %the last mode has extra targets
            num_of_targets_per_mode{num_of_modes}=num_of_extra_targets+...
                num_of_targets_per_mode{num_of_modes};
            
            %targets
            for i=1:num_of_modes                
                %target set is a subset of the mode set
                target_set{i}=[mode_set_center{i}-((mode_set{i}(:,2)-...
                    mode_set{i}(:,1))/2)*target_r...
                    mode_set_center{i}+((mode_set{i}(:,2)-...
                    mode_set{i}(:,1))/2)*target_r];
            end
            
            %scaling
            a1 = 0.9;
            b1 = 0.3;
            %offset
            a2 = -0.4;
            b2 = 0.4;
            for ii=1:num_of_extra_targets
                %choose a subset of the mode set
                target_r = (b1-a1).*rand(1,1) + a1;
                target_shift = (b2-a2).*rand(2,2) + a2;
                target_set{num_of_modes+ii}=[mode_set_center{fMode}-...
                    ((mode_set{fMode}(:,2)-mode_set{fMode}(:,1))/2)*target_r...
                    mode_set_center{fMode}+((mode_set{fMode}(:,2)-...
                    mode_set{fMode}(:,1))/2)*target_r]+target_shift;
            end
        end
    end
end
\end{lstlisting}

\bibliographystyle{splncs}
\bibliography{bibJournal}

\begin{thebibliography}{10}
\providecommand{\url}[1]{\texttt{#1}}
\providecommand{\urlprefix}{URL }

\bibitem{BDDbook}
Akers, S.: Binary decision diagrams. Computers, IEEE Transactions on  C-27(6),
  509--516 (June 1978)

\bibitem{fragment2}
Alur, R., La~Torre, S.: Deterministic generators and games for {LTL} fragments.
  ACM Trans. Comput. Logic  5(1),  1--25 (Jan 2004)

\bibitem{intro4}
Alur, R., Henzinger, T.A., Vardi, M.Y.: Theory in practice for system design
  and verification. ACM SIGLOG News  2(1),  46--51 (Jan 2015),
  \url{http://doi.acm.org/10.1145/2728816.2728827}

\bibitem{fragment1}
Asarin, E., Maler, O., Pnueli, A., Sifakis, J.: Controller synthesis for timed
  automata (1998)

\bibitem{mode_targetweb}
Balkan, A.: \url{http://www.cyphylab.ee.ucla.edu/modetargettac}

\bibitem{ADHSBalkan}
Balkan, A., Vardi, M., Tabuada, P.: Controller synthesis for mode-target games.
  In: Proceedings 5th IFAC Conference on Analysis and Design of Hybrid Systems.
  ADHS (2016)

\bibitem{GR1}
Bloem, R., Jobstmann, B., Piterman, N., Pnueli, A., Saar, Y.: Synthesis of
  reactive(1) designs. J. Comput. Syst. Sci.  78(3),  911--938 (May 2012)

\bibitem{Browne}
Browne, A., Clarke, E., Jha, S., Long, D., Marrero, W.: An improved algorithm
  for the evaluation of fixpoint expressions. Theoretical Computer Science
  178(1-2),  237 -- 255 (1997)

\bibitem{GeneralizedBuchi}
Dziembowski, S., Jurdzinski, M., Walukiewicz, I.: How much memory is needed to
  win infinite games? In: Logic in Computer Science, 1997. LICS '97.
  Proceedings., 12th Annual IEEE Symposium on. pp. 99--110 (Jun 1997)

\bibitem{GeneralizedRabin}
Ehlers, R.: Generalized {R}abin(1) synthesis with applications to robust system
  synthesis. In: Bobaru, M., Havelund, K., Holzmann, G.J., Joshi, R. (eds.)
  NASA Formal Methods, Lecture Notes in Computer Science, vol. 6617, pp.
  101--115. Springer Berlin Heidelberg (2011)

\bibitem{emersonlei}
Emerson, E.A., Lei, C.L.: {Efficient Model Checking in Fragments of the
  Propositional Mu-Calculus (Extended Abstract)}. In: Proceedings of the First
  Annual {IEEE} Symposium on Logic in Computer Science (LICS). pp. 267--278
  (Jun 1986)

\bibitem{ISOstandard}
{ISO 15622:2010 (E)}: Intelligent transport systems -- adaptive cruise control
  systems -- performance requirements and test procedures. Tech. rep.,
  International Organization for Standardization (2010)

\bibitem{jyoHSCC}
Jin, X., Deshmukh, J.V., Kapinski, J., Ueda, K., Butts, K.: Powertrain control
  verification benchmark. In: Proceedings of the 17th International Conference
  on Hybrid Systems: Computation and Control. pp. 253--262. HSCC '14, ACM, New
  York, NY, USA (2014)

\bibitem{Kesten}
Kesten, Y., Piterman, N., Pnueli, A.: Bridging the gap between fair simulation
  and trace inclusion. Information and Computation  200(1),  35 -- 61 (2005)

\bibitem{application1}
Kress-Gazit, H., Fainekos, G.E.: Where's {W}aldo? sensor-based temporal logic
  motion planning. In: in IEEE International Conference on Robotics and
  Automation. pp. 3116--3121 (2007)

\bibitem{fragment3}
K{\v{r}}et{\'i}nsk{\'y}, J., Esparza, J.: Deterministic Automata for the
  (F,G)-Fragment of LTL, pp. 7--22. Springer Berlin Heidelberg, Berlin,
  Heidelberg (2012)

\bibitem{application2}
L., J., Ozay, N., Topcu, U., Murray, R.: Synthesis of reactive switching
  protocols from temporal logic specifications. Automatic Control, IEEE
  Transactions on  58(7),  1771--1785 (July 2013)

\bibitem{Pessoa}
Mazo, M., Davitian, A., Tabuada, P.: Pessoa: A tool for embedded controller
  synthesis. In: Proceedings of the 22Nd International Conference on Computer
  Aided Verification. pp. 566--569. CAV'10, Springer-Verlag, Berlin, Heidelberg
  (2010)

\bibitem{ACCpaper2}
Nilsson, P., Hussien, O., Balkan, A., Chen, Y., Ames, A., Grizzle, J., Ozay,
  N., Peng, H., Tabuada, P.: Correct-by-construction adaptive cruise control:
  Two approaches. IEEE Transactions on Control Systems Technology  (2016)

\bibitem{intro2}
Seshia, S.A.: New frontiers in formal methods: Learning, cyber-physical
  systems, education, and beyond. CSI Journal of Computing  2(4),  R1:3--R1:13
  (June 2015)

\bibitem{intro3}
Sifakis, J.: System design automation: Challenges and limitations. Proceedings
  of the {IEEE}  103(11),  2093--2103 (2015),
  \url{http://dx.doi.org/10.1109/JPROC.2015.2484060}

\bibitem{Sistla}
Sistla, A.: Safety, liveness and fairness in temporal logic. Formal Aspects of
  Computing  6(5),  495--511 (1994)

\bibitem{Streett}
Streett, R.S.: Propositional dynamic logic of looping and converse. In:
  Proceedings of the Thirteenth Annual ACM Symposium on Theory of Computing.
  pp. 375--383. STOC '81, ACM, New York, NY, USA (1981),
  \url{http://doi.acm.org/10.1145/800076.802492}

\bibitem{Paulobook}
Tabuada, P.: Verification and control of hybrid systems: a symbolic approach.
  Springer (2009)

\bibitem{tarski}
Tarski, A.: A lattice-theoretical fixpoint theorem and its applications.
  Pacific J. Math.  5(2),  285--309 (1955),
  \url{http://projecteuclid.org/euclid.pjm/1103044538}

\bibitem{intro1}
Vardi, M.Y.: From Verification to Synthesis, pp. 2--2. Springer Berlin
  Heidelberg, Berlin, Heidelberg (2008),
  \url{http://dx.doi.org/10.1007/978-3-540-87873-5_2}

\bibitem{nuclearpowerplant}
Venz, H., Ruhle, W., Kysela, J.: Start-up and shutdown practices in {BWR}s as
  well as in primary and secondary circuits of {PWR}s, {VVER}s and {CANDU}s.
  Tech. rep., ANT International (2009)

\bibitem{wolf}
Wolff, E., Topcu, U., Murray, R.: Efficient reactive controller synthesis for a
  fragment of linear temporal logic. In: IEEE International Conference on
  Robotics and Automation (ICRA). pp. 5033--5040 (May 2013)

\end{thebibliography}

\end{document}